\documentclass[a4paper,reqno,11pt,oneside]{amsart}
\usepackage{amssymb,amsmath, amsfonts, amsthm,color}
\usepackage{graphicx}
\usepackage{mathrsfs}
\usepackage{appendix}

\usepackage{latexsym}
\usepackage{color} 
\usepackage[dvips]{epsfig}
\usepackage[colorlinks]{hyperref}
\usepackage{comment}
\usepackage{geometry}
\hypersetup{
 colorlinks,
 citecolor=red,
 urlcolor=black
}

\newtheorem{theorem}{Theorem}[section]
\newtheorem{lemma}[theorem]{Lemma}

\newtheorem{proposition}[theorem]{Proposition}

\newtheorem{remark}[theorem]{Remark}
\numberwithin{equation}{section}

\title[New solutions for the Lane-Emden problem in planar domains]{New solutions for the Lane-Emden problem\\ in planar domains}

\author[L. Battaglia]{Luca Battaglia}
\address{Luca Battaglia, Dipartimento di Matematica e Fisica, Universit\`a degli Studi Roma Tre, Via della Vasca Navale 84, 00146 Roma (Italy)}
\email{luca.battaglia@uniroma3.it}
\author[I. Ianni]{Isabella Ianni}
\address{Isabella Ianni, Dipartimento di Scienze di Base e Applicate per l'Ingegneria, Sapienza Universit\`a di Roma, Via Antonio Scarpa 10, 00161 Roma (Italy)}
\email{isabella.ianni@uniroma1.it}
\author[A. Pistoia]{Angela Pistoia}
\address{Angela Pistoia, Dipartimento di Scienze di Base e Applicate per l'Ingegneria, Sapienza Universit\`a di Roma, Via Antonio Scarpa 10, 00161 Roma (Italy)}
\email{angela.pistoia@uniroma1.it}

\begin{document}

\everymath{\displaystyle}
\allowdisplaybreaks

\date\today


\thanks{ L. Battaglia is partially supported by MUR-PRIN-2022AKNSE4 ``Variational and Analytical aspects of Geometric PDEs''. I. Ianni and A. Pistoia are partially supported by the MUR-PRIN-20227HX33Z ``Pattern formation in nonlinear phenomena''. L. Battaglia and A. Pistoia are also partially supported by INDAM-GNAMPA project ``Problemi di doppia curvatura su varietà a bordo e legami con le EDP di tipo ellittico''.
{L. Battaglia and I. Ianni are also partially supported by the INDAM-GNAMPA project ``Proprietà qualitative delle soluzioni di equazioni ellittiche''. }
}

\begin{abstract}
We consider the Lane-Emden problem
$$-\Delta u=|u|^{p-1} u\ \hbox{in}\ \Omega,\ u=0\ \hbox{on} \ \partial\Omega,$$
where $\Omega\subset\mathbb R^2$ is a smooth bounded domain. When the exponent $p$ is large, the existence and multiplicity of solutions strongly depend on the geometric properties of the domain, which also deeply affect their qualitative behaviour. Remarkably, a wide variety of solutions, both positive and sign-changing, have been found when $p$ is sufficiently large. In this paper, we focus on this topic and find new sign-changing solutions that exhibit an unexpected concentration phenomenon as $p$ approaches $+\infty.$
\end{abstract}

\maketitle\

\section{Introduction}\
We consider the Lane-Emden problem
\begin{equation}
\label{eqp}
 -\Delta u=|u|^{p-1}u\ \hbox{in}\ \Omega,\quad u=0\ \hbox{on}\ \partial\Omega.
 \end{equation}
where $\Omega\subset\mathbb R^2$ is a smooth bounded domain and $p>1$.\\

The existence of both positive and sign-changing solutions for \eqref{eqp} may be found by standard variational methods. Despite the seemingly simple appearance of this equation, the set of solutions can be very rich. In particular, when the exponent $p$ is large the existence and multiplicity of solutions strongly depends on the geometric properties of the domain which also deeply affect their qualitative behaviour. Remarkably, a wide variety of solutions both positive and sign-changing to \eqref{eqp} can be found when $p$ is sufficiently large. In this paper we focus on this topic and we find a new sign-changing solution which exhibit an unexpected concentration phenomena when $p$ approaches $+\infty.$
\\


Recently in \cite{DeMarchisIanniPacellaJFPTA} the behavior of any positive solution $u$ of \eqref{eqp}, in any smooth bounded domain $\Omega$, has been fully characterized when $p$ is large enough, under uniform energy bounds. Indeed in \cite{DeMarchisIanniPacellaJFPTA} (see also \cite{DeMarchisGrossiIanniPacellaLinfty,Thizy, DeMarchisIanniPacellaJEMS, AdimurthiGrossi,RenWeiTAMS1994,RenWeiPAMS1996} for partial results) it has been shown that when $p\rightarrow +\infty$ the positive solutions necessarily concentrate at a finite number of interior points in $\Omega$ and vanish locally uniformly outside the
concentration set. Around each concentration point moreover, the local maximum value of the solution converges to $\sqrt e$ (we stress that positive solutions are also uniformly bounded in $p$, see \cite{KamburovSirakov}), the concentration is simple and the asymptotic profile of the solution (suitably rescaled around the concentration point) is given by the \emph{bubble} $U$ which is a radial solution of the Liouville equation $-\Delta U=e^U$ in the whole $\mathbb R^2$. 
{Furthermore, if the solution concentrates at the $\kappa$ points
$ \xi_1,\dots,\xi_\kappa$, called {\em peaks}, then (see \cite{DeMarchisIanniPacellaJFPTA}) the total energy of the solution converges to $\kappa\cdot8\pi e$ and $(\xi_1,\dots,\xi_\kappa)$ must be a critical point of the {\em Kirchoff-Routh} function 
\begin{equation}\label{kirko}\mathscr K(x_1,\dots,x_\kappa):=\sum\limits_i
\Lambda_i^2 H(x_i,x_i)-\sum\limits_{i\not=j}\Lambda_i\Lambda_j G(x_i,x_j),\end{equation} 
where all the parameters $\Lambda_i$'s are equal to $+1.$ Here $H(x,x)$ is the Robin function, being
 {$H(x,y)$} the regular part of the Green's function of $-\Delta$ with Dirichlet boundary conditions, i.e. {$G(x,y)=-\frac1{2\pi}\ln {|x-y|}-H(x,y)$}.\\
Conversely, in \cite{EspositoMussoPistoiaJDE2006} the authors proved that if $(\xi_1,\dots,\xi_\kappa)$ is a {non-degenerate critical point} of the {Kirchoff-Routh} function
then if $p$ is large enough the problem \eqref{eqp} has a solution which concentrate at $\xi_1,\dots,\xi_\kappa$. It is worthwhile to point out that for generic domains $\Omega$ the {Kirchoff-Routh} function is a Morse function (see \cite{bartsch-miche-pis,miche-pis}).
Therefore, the existence of concentrating positive solutions turns to be linked with the existence of critical points of the {Kirchoff-Routh} function. By \cite{grossi-takahashi} we deduce that in a convex domain there are no $\kappa-$peak solutions whenever $\kappa\ge2$, and furthermore the (one-peak) solution is unique for large $p$ (see \cite{DeMarchisGrossiIanniPacella, GrossiIanniLuoYan,LinMANMATH1994}). On the other hand if the domain is not simply connected by \cite{delpino-kow-musso} we deduce that for any $\kappa\ge1$ it does exist a $\kappa-$peak solution. By \cite{egp} we also deduce that there exists a $\kappa-$peaks solution in a simply connected domain whose shape is a $\kappa-$dumbell (namely it is obtained by connecting $\kappa$ disjoint domains with small channels). }
\\

Concerning sign-changing solutions, it is known (see \cite{DeMarchisIanniPacellaJEMS}) that also in this case, under uniform energy bounds in $p$, nodal solutions are uniformly bounded in
$p$, concentrate at a finite number of points and converge to zero locally uniformly outside the
concentration set as $p\rightarrow +\infty$, although a complete characterization of the concentration behavior is not yet available. In particular, the concentration may be not simple, and different phenomena may appear in the limit as $p\rightarrow +\infty$, depending on the asymptotic behavior of the nodal line of the solution (see \cite{GrossiGrumiauPacellaRadial, IanniSaldana, GrossiGrumiauPacellaLowEnergy, DeMarchisIanniPacellaJEMS}). \\

Examples of sign-changing solutions with simple concentration, positive at an interior point and negative at another, can be found in \cite{ GrossiGrumiauPacellaLowEnergy}, where it is shown that the asymptotic profile (after scaling) around each concentration point is again given by the radial bubble $U$, each bubble carrying the same amount of energy, that is $8\pi e$. Moreover, in \cite{EspositoMussoPistoiaPLMS}, the authors build 
solutions concentrating at $\kappa$ different points $\xi_1,\dots,\xi_\kappa$ provided that $(\xi_1,\dots,\xi_\kappa)$ is a non-degenerate critical point of the {Kirchoff-Routh} function \eqref{kirko} when the $\Lambda_i$'s are equal to $\pm1$. More precisely, if $\Lambda_i=+1$, the peak $\xi_i$ is positive, whereas if $\Lambda_i=-1$, the peak $\xi_i$ is negative. In particular, they prove that in any domain there exists at least two pairs of sign-changing solutions, each with one positive and one negative peak. Multiplicity of solutions with peaks of different signs has been obtained in \cite{bartsch-pistoia-weth}. In particular, if the domain is axially symmetric for any $\kappa\ge1$, there exists a sign-changing solution with $\kappa$ negative peaks 
 alternating with $\kappa$ positive peaks. All the solutions considered above have {\em simple and isolated} peaks.
 \\
 
 A different asymptotic behavior belongs instead to the nodal radial solutions. In \cite{IanniSaldana} (see also \cite{GrossiGrumiauPacellaRadial}) it has been proved that, when $\Omega$ is a ball, any radial sign-changing solution concentrates only at the origin and the concentration is \emph{non-simple}. Indeed a \emph{tower} of $\kappa$ \emph{different bubbles} with alternating sign and all centered at the origin appears, $\kappa$ being the number of the nodal regions of the solution. Specifically, all the $\kappa-1$ nodal circles shrink to zero at diminishing speeds from the innermost circle to the outermost, and the restriction of the solution to each nodal region, suitably rescaled, converges to a radial solution $U_i$ of the singular Liouville equation
 $-\Delta U_i=|x|^{c_i}e^{U_i} $ in $\mathbb R^2$, where $c_i$, $i=1,\ldots, \kappa$, are explicitly given (and $c_0=0$, namely $U_0=U$). Hence the solution looks like the superposition of $k$ different bubbles $U_i$, with alternating signs, concentrating at the origin at different scales as $p\rightarrow +\infty$. This \emph{tower of bubbles} phenomenon happens also in non-radial cases: in \cite{DeMarchisIanniPacellaJEMS} a tower of two bubbles for solutions in symmetric domains more general than balls are shown. Recently in \cite{pi-ric}, the authors construct solutions to \eqref{eqp} for $p$ large enough whose profile resembles a tower of $\kappa$ bubbles, for any $\kappa\geq 2$, with alternating sign and all centered at the origin, in a more general but still symmetric domain. We emphasize that in all these examples, the origin is an {\em isolated and non-simple} peak. \\

The question whether there are other types of concentration phenomena for sign-changing solutions of \eqref{eqp} different from the tower of bubbles, naturally arises. This question is still unanswered, so far.\\

Recently, in \cite{GladialiIanni} in the case when $\Omega$ is a ball, unexpected sign-changing solutions are found which are 
invariant by the action of the discrete group of symmetries given by the reflection with respect to an axis and the rotations of angle $\frac{2\pi}{k}$, with $k=3,4,5$. These solutions are non-radial, and their nodal line does not intersect the boundary of the domain, hence they have been called \emph{quasi-radial}. Numerical approximations later done in \cite{FazekasPacellaPlum} seem to suggest that these solutions concentrate only at the center of the ball in the limit as $p\rightarrow +\infty$, but they should not behave like a tower of bubbles. {Actually, the origin looks like a \emph{non-simple} and \emph{non-isolated} concentration point. } More precisely the profile of these solutions could be the one of a \emph{cluster of bubbles}. 
The expression \emph{cluster of bubbles} is commonly referred to a superposition of bubbles centered and concentrating at a finite number of different points, when these points collapse towards each other, with a speed of collision slower than the speeds of concentration.\\ 

This naturally leads to the question:

{\em is it possible to find a solution to \eqref{eqp} whose profile resembles a \underline{cluster of bubbles}?}\\

In this paper we give a positive answer to this open question. 
Indeed we build sign-changing solutions to \eqref{eqp}
 having a \emph{non-simple} concentration point in $\Omega$ (which we may assume to be the origin w.l.g.) and vanishing elsewhere in the limit as $p\rightarrow +\infty$, and whose asymptotic behavior is the one of a \emph{cluster of bubbles}.
Specifically the solutions that we found look like the superposition of two different profiles (at a first order): one standard radial bubble $U$ concentrating positively at the origin, and a second \emph{different profile}, non-radial but with the discrete $k$-symmetry around the origin itself, concentrating negatively at $k$ different symmetric points with a slower speed, furthermore these $k$ points collapse towards the origin with a speed which is slower than both the speeds of concentration of the two profiles.
We obtain this result under some symmetry assumptions on the domain $\Omega$ (which keep the origin fixed) and with a natural restriction on $k$.\\

In order to state it rigorously 
let us assume $0\in\Omega$ and $\Omega$ to be invariant with respect to reflection across the $x_1$ axis:
\begin{equation}\label{firstAssumptionOmega}x=(x_1,x_2)\in\Omega\qquad\iff\qquad\overline x:=(x_1,-x_2)\in\Omega.\end{equation}
We also ask $\Omega$ to be a {\em $k$-symmetric domain}, for some integer $k$, that is invariant under a rotation of angle $\frac{2\pi}k$ with respect to the origin:
\begin{equation}\label{secondAssumptionOmega}x\in\Omega\,\,\,\,\,\iff\,\,\,\,\,e^{\mathrm i\frac{2\pi}k}x:=\left(x_1\cos\frac{2\pi}k-x_2\sin\frac{2\pi}k,x_1\sin\frac{2\pi}k+x_2\cos\frac{2\pi}k\right)\in\Omega.\end{equation}
Obviously the ball satisfies both the assumption \eqref{firstAssumptionOmega} and \eqref{secondAssumptionOmega}.\\
Let us also introduce the following space of {\em $k-$symmetric functions}:
\begin{equation}\label{def:H10sim}H^1_{0,k}(\Omega):=\left\{u\in H^1_0(\Omega):\,u(x)=u\left(\overline x\right)=u\left(e^{\mathrm i\frac{2\pi}k}x\right)\,\mbox{for a.e. }x\in\Omega\right\}.\end{equation}

 Our main result reads as follows.

\begin{theorem}\label{teop} Let $\Omega\subset\mathbb R^2$ be a smooth bounded domain such that $0\in\Omega$ which satisfies assumption \eqref{firstAssumptionOmega}. Assume $k=4,5.$ If $\Omega$ is $k-$symmetric (see \eqref{secondAssumptionOmega}), then there exists $p_k>1$ such that, for $p>p_k$ there exists a nodal, $k$-symmetric solution $u_{k,p}\in H^1_{0,k}(\Omega)$ to \eqref{eqp} such that:
\begin{equation}\label{solutionInTheorem}u_{k,p}=\tau_{p}\left(PU_{1,p}-\eta_pPU_{2,p}\right)+\ o\left(\frac{1}{p}\right) \quad\mbox{ in } {L^{\infty}_{loc}(\overline\Omega\setminus 0)}\end{equation}
as $p\rightarrow +\infty$, where $PU_{1,p}$ and $PU_{2,p}$ are the projections on $H^{1}_{0}(\Omega)$ (see \eqref{proie}) of
\[U_{1,p}(x):=\log \frac{8\alpha_{p}^2}{\left(|x|^2+\alpha_{p}^2\right)^2},\ \hbox{which solves}\ -\Delta U_{1,p}=e^{U_{1,p}}\ \hbox{in}\ \mathbb R^2 \]
and \[U_{2,p}(x):=\log \frac{8k^2\beta_{p}^{2k}}{\left(\left|x^k-\rho_{p}^k\right|^2+\beta_{p}^{2k}\right)^2},\ \hbox{which solves}\ -\Delta U_{2,p}=|x|^{2k-2}e^{U_{2,p}}\ \hbox{in}\ \mathbb R^2. \]
Here we agree that the point $\rho_p^k:=\left(\rho_p^k,0\right)\in\mathbb C$. All the real parameters, $\eta_p$,
$\tau_{p}$, $\alpha_{p}$, $\beta_{p}$, $\rho_{p}$ are positive and as $p\rightarrow +\infty$ satisfy
$$\tau_{p}\sim \displaystyle\frac{ t_k}p,\quad \eta_p\sim \eta_{\infty,k}:=\frac{2}{k-1},$$ $$\alpha_{p}\sim A_ke^{- \displaystyle{( a_k+o(1))}p},
 \quad\beta_{p}\sim B_ke^{- \displaystyle{( b_k+o(1))}p}
 \ \hbox{ and }\ 
 \rho_{p}\sim R_ke^{-\displaystyle {( r_k+o(1))}p},
 $$
 where $ A_k,$ $ B_k$ and $ R_k$ are positive constants which only depend on $k$, and \\
 \begin{equation}\label{parax}\begin{aligned}& 
 t_k:= 
 \sqrt e \left( \frac{2}{k-1}\right)^{-\frac{4k}{(k+1)^2}}\\
 & a_k:= 
\frac14\left(1+\frac{8k}{(k+1)^2}\ln\frac{k-1}2\right)\\
 & b_k:=
\frac1{4k}\left(1+2\frac{(k-1)^2}{(k+1)^2}\ln\frac{k-1}2\right)\\
 & r_k:=
 \frac{k-1}{(k+1)^2}\ln\frac{k-1}2.\end{aligned}\end{equation}
 Moreover, as $p\to+\infty$, the positive part and the negative part of $u_{k,p}$ satisfy
 $$
p\int_\Omega{(u_{k,p}^+)}^{p+1}\sim 8\pi e\left(\frac{k-1}2\right)^\frac{8k}{(k+1)^2}\quad \hbox{and}\quad p\int_\Omega(u_{k,p}^-)^{p+1}\sim 8\pi ek\left(\frac2{k-1}\right)^\frac{2(k-1)^2}{(k+1)^2}.$$
Finally
$$p|u_{k,p}|^{p+1} \to0\qquad\mbox{in }L^\infty_{\mathrm{loc}}\left(\overline\Omega\setminus\{0\}\right).$$
\end{theorem}

Above we agree that $f\sim g$ means $\lim\limits_{p\to+\infty}\frac{f(p)}{g(p)}\in\mathbb R\setminus\{0\}.$\\

The solutions described in the theorem are rather delicate to be captured, indeed the proof of Theorem \ref{teop} relies on a particularly challenging Ljapunov-Schmidt procedure. A first issue is the choice of the correct ansatz for the approximation. In the proof, we need to add to the first order term given by the two bubbles $U_{1,p}$ and $U_{2,p}$ in \eqref{solutionInTheorem}, two extra terms up to the third order, involving four other profiles $V_{1,p}$, $V_{2,p}$, $W_{1,p}$, $W_{2,p}$ (see Subsection \ref{subsection:ansatzTotal} for the complete ansatz). This higher order expansion is essentially due to the fact that the pure power $u^p$ is not close enough to the exponential $e^u$ as $p\to +\infty$. Secondly, for the error to be small, we must precisely select all the parameters in the involved profiles (see \eqref{parax}), and this turns out to be quite subtle (see Section \ref{3} for the details). As a matter of fact, the estimate of the error size is
 the most delicate and technical part of the paper, for which we must distinguish among the regime close to the positive peak which is the origin, the one close to the $k$ negative peaks, which are collapsing to the origin, and far away from the peaks. 
\\

Now, let us make some comments on the result.\\

\begin{remark} \rm The parameters satisfy \eqref{parax} and so it is straightforward to check that the choice $k=4$ or $k=5$ 
ensures \[a_k> b_k> r_k>0.\] Such a condition implies that the concentration rate $\alpha_p$ of the positive bubble is faster than the concentration rate $\beta_p$ of the negative bubble, and both are faster than the rate $\rho_p$ at which the negative peaks (i.e. the $k$th-roots $z_{j,p}:=\rho_pe^{i\frac{2\pi j}{k}}\in\mathbb C$, $j=1,\ldots, k$ of the complex number $\rho_p^k$) are approaching the positive one (i.e. the origin). \end{remark}

\begin{remark}\rm 
 The clustering behaviour requires in particular that {the speed at which the negative peaks approach the positive one is slower than the rate of concentration of the negative bubble}, i.e. in \eqref{parax} $ b_k> r_k>0$, namely
$$ 
 \frac12-\frac{k-1}{k+1}\log \frac{k-1}2>0 \quad \hbox{ and }\quad (k-1)\ln{k-1\over2}>0.
 $$
 A direct computation shows that this is possible whenever
 $3<k< k_0$ where $k_0\sim 5.187$ is the only solution to the equation
\begin{equation}\label{kzero}
\frac12-\frac{k_0-1}{k_0+1}\log \frac{k_0-1}2=0.
\end{equation}
This forces the choice $k=4, 5$ which complies with the result found in \cite{GladialiIanni}. We notice that in \cite{GladialiIanni} the case $k=3$ is also considered, in the next remark we comment on the reason why this case is not covered by Theorem \ref{teop}. \\
We strongly believe that if $k\ge6$ the clustering phenomena do not occur. The proof should require a subtle understanding of the interactions between the bubbling profiles, and rely on sharp pointwise estimates around all the peaks. This subject could be a challenging topic for future research.\end{remark}

 \begin{remark} \rm In this theorem we consider only the values $k=4,5$. It is immediate to verify that if $k=3$, then $r_k=0$ in \eqref{parax}, which is not possible. 
 We strongly believe that the case $k=3$ can be managed, provided that all parameters are chosen up to the second order. 
 Clearly, this makes the whole reduction procedure much more difficult. However,
 we believe that all the heavy technicalities required to address this problem could also be used to produce solutions with a non-simple blow-up point for the widely studied sinh-Poisson equation, as predicted in \cite{esposito-wei} and eventually not found due to an error in the reduced energy (i.e. the leading term vanishes).
All this could be an interesting line of research.
\end{remark}
 
\begin{remark} \rm
The total energy of the solutions found in Theorem \ref{teop} as $p\to+\infty$ satisfies
$$p\int_\Omega|u_{k,p}|^{p+1} \sim 8\pi e\left(\frac{k+1}{k-1}\right)^2\left(\frac{k-1}2\right)^\frac{8k}{(k+1)^2}=:E(k).$$
On the other hand, when $\Omega=B_1(0)$ is the unit ball, the least energy nodal radial solution $u_p$ was shown in \cite[Theorem 2.4]{IanniSaldana} (see also \cite{GrossiGrumiauPacellaRadial}) to satisfy
$$p\int_\Omega|u_p|^{p+1}\sim E(k_0),$$
with $k_0$ as in \eqref{kzero}. Since the real function $E(k)$ has its strict maximum in $k=k_0$, Theorem \ref{teop} shows in particular that the least energy nodal $k$-symmetric solution is \emph{not} radial for $k=4,5$ for $p$ large. This is consistent with the analysis in \cite[Theorem 1.3]{GladialiIanni}.
\end{remark}

\begin{remark} \rm We conjecture that a solution 
resembling a similar {\em cluster of bubbles} as $p$ approaches $+\infty$ should exist in a general domain $\Omega$ without any symmetry assumption, with the non-simple concentration point being a possibly non-degenerate critical point of the Robin's function. The proof is undoubtedly challenging and cannot be immediately deduced by our result. Our symmetry assumptions simplify both the construction of the ansatz and the study of the linear theory. \end{remark}

\begin{remark} \rm
In higher dimension (i.e. 
$\Omega\subset\mathbb R^n$ and $n\ge3$), solutions to \eqref{eqp} having simple and isolated positive and/or negative blow-up peaks when the exponent $p$ approaches the critical Sobolev exponent from below, i.e. $p=(n+2)/(n-2)-\epsilon$ with $\epsilon\to 0^+,$ have been found in \cite{bahri-li-rey,bartsch-miche-pistoia,bartsch-daprile-pistoia1,bartsch-daprile-pistoia2}. Moreover, sign-changing solutions whose profile resembles 
towers of bubbles with alternating sign have been built in \cite{musso-pistoia,pistoia-weth}.
In this case, the profiles of the bubbles in the towers, as well as the energies they carry, are all the same, unlike the $2$-dimensional case (\cite{IanniSaldana,GrossiGrumiauPacellaRadial,DeMarchisIanniPacellaJEMS,pi-ric}).
The question whether, in higher dimension, there are other types of non-simple concentration phenomena for nodal solutions of \eqref{eqp}, different from both the simple blow-ups and from the tower of bubbles, for instance cluster of bubbles, naturally arises and remains still unanswered, as far as we know.\end{remark}

The paper is organized as follows.\\
In Section \ref{2}, we describe the ansatz for the solution we are looking for. As already observed, we need to add to the first order term (described in Subsection \ref{subsection:ansatzFirst}), two extra terms up to the third order (introduced in Subsection \ref{subsection:ansatzSecond}). Here we follow some ideas introduced in \cite{EspositoMussoPistoiaJDE2006,pi-ric}. The final ansatz is given in Subsection \ref{subsection:ansatzTotal}. \\
Next, in Section \ref{3} we shall compute the error size.  The central result of the section is Proposition \ref{r}. In order to derive it, we need first to 
adjust all the parameters in the ansatz in order to have a sufficiently good 
approximation of the solution; this delicate part is done in Subsection \ref{subsection:parametersineta}. Some preliminary estimates are also provided in Subsections \ref{subsection:preliminaryEta}, \ref{subsection:ansatzStime} and \ref{upversuseu}. Then, we compute the error size, first close to the positive peak in Subsection \ref{Section:errorOrigin}, then close to the negative peaks in Subsection
\ref{Section:errorRoots} and, at last, far away from the peaks in Subsection
 \ref{Section:errorFaraway}. In Subsection \ref{Section:errorTotale} we finally measure the global size of the error in an appropriate weighted norm, getting  Proposition \ref{r}. \\ The choice of the weighted norm is crucial also in the study of the linear theory, which is developed in Section \ref{4}. The main result of this section is the invertibility stated in Proposition \ref{linear}. We stress that sections \ref{3} and \ref{4} contain the most delicate and technical parts of the paper, indeed careful estimates are needed around the positive peak which is the origin, and around the negative peaks which are $k$ points that are collapsing to the origin.\\
In Section \ref{5} we complete the proof of Theorem \ref{teop}. First, via a classical contraction mapping argument (see Subsection \ref{subsection:contraction}), we reduce \eqref{eqp} to a finite dimensional problem whose solutions are critical points of a function of one real variable (see 
 Subsection \ref{subsection:reduced}), which actually has a minimum point (see Subsection \ref{subsection:minimo}). Finally, all the properties owned by the solution built in this way and listed in Theorem \ref{teop} are proved in Subsection \ref{subsection:fine}.\\

 { \em Notations.} In this paper we will use $c$ or $C$ to denote a positive constant whose value may change from line to line.
 The symbol $g= O(f)$ will be used to mean $|g|\le c|f|$ when comparing two quantities. Moreover, the symbol $g= O_{\leq 0}(f)$ means $-c|f|\le g\le 0.$

\section{The ansatz}\label{2}

{In this section we introduce the main building blocks which will be used to construct the solution we are looking for, hence make the general ansatz for the solution. We conclude the section by introducing the classical reduction process. }

\subsection{The first order term of the ansatz: the bubbles}\label{subsection:ansatzFirst}
We consider the \emph{standard bubble}
\begin{equation}\label{bubble}
U(x):=\log \frac8{\left(|x|^2+1\right)^2},
\end{equation}
which solves
\begin{equation}\label{liouvilleq}
-\Delta U(x)=e^{U(x)},\qquad x\in\mathbb R^2.
\end{equation}
We know that, for any $\alpha>0$, equation \eqref{liouvilleq} is also solved by
$$U_1(x):=U\left(\frac x\alpha\right)-2\log \alpha=\log \frac{8\alpha^2}{\left(|x|^2+\alpha^2\right)^2}.$$
We also know that, for $\beta,\rho>0,k\in\mathbb N$, the function
$$U_2(x):=U\left(\frac{x^k-\rho^k}{\beta^k}\right)+\log \left(k^2\right)-2k\log \beta=\log \frac{8k^2\beta^{2k}}{\left(\left|x^k-\rho^k\right|^2+\beta^{2k}\right)^2}$$
solves
\begin{equation}\label{liouvillesing}
-\Delta U_2(x)=|x|^{2k-2}e^{U_2(x)},\qquad x\in\mathbb R^2.
\end{equation}
Both $U_{1}$ and $U_{2}$ will be used in our ansatz for the solution, to approximate it at a first order (see Subsection \ref{subsection:ansatzTotal}).

\

Here and in the rest of the paper we use the complex notation to mean $$x^k=\left(\Re\left((x_1+\mathrm ix_2)^k\right),\Im\left((x_1+\mathrm ix_2)^k\right)\right)\in\mathbb R^2,\qquad\mbox{for }x=(x_1,x_2)\in\mathbb R^2.$$
We also agree to identify any positive real number $t>0$ with its corresponding point on the positive real half-line
$$t=t+\mathrm i\cdot0=(t,0)\in\mathbb R^2.$$
For any $y\in\mathbb R^2$, we will write $\sqrt[k]y$ to mean \emph{any} root $x\in\mathbb R^2$ of the equation $x^k=y$. 
It is useful to point out that since we will consider symmetric functions $u\in H^1_{0,k}(\Omega)$ (see the definition in \eqref{def:H10sim}) in particular 
$u(x)=u(\widetilde x)$ if $x^{k}={\widetilde x}^{k}$, hence the choice between one root and another will make no difference.

\subsection{A refinement of the ansatz: the second and the third order terms}\label{subsection:ansatzSecond}

In order to have an error term sufficiently small {(which will be done in the next subsection)} we need to get a better approximation of the solution, namely we need to add some extra lower-order terms in the spirit of \cite{EspositoMussoPistoiaJDE2006} and \cite{pi-ric}. 

\

Hence we consider entire solutions to the following problems in the whole $\mathbb R^2$:
\begin{equation}
\label{vw}\Delta V+e^UV=e^Uf_1(U)\quad \hbox{ and }\quad \Delta W+e^UW=e^Uf_2(U,V),
\end{equation}
where $U$ is the bubble defined in \eqref{bubble}, 
$$f_1(u):=\frac{u^2}2\quad \hbox{ and }\quad f_2(u,v):=uv-\frac{u^3}2-\frac12\left(v-\frac{u^2}2\right)^2.$$
{The choice of $f_1$ and $f_2$ is done in order to control the error. We refer to Section \ref{3} for a complete discussion about this, here we only stress that it is strictly linked to the use of the Taylor expansion in Lemma \ref{taylor} (see in particular Subsections \ref{Section:errorOrigin} and \ref{Section:errorRoots}, where Lemma \ref{taylor} is needed).}

\

The existence of suitable radial $V,W$ with logaritmic behavior at infinity follows as in \cite{EspositoMussoPistoiaJDE2006,pi-ric} (see also \cite{ci}). In particular, $V$ has the following asymptotic behavior, as $|x|$ goes to $+\infty$:
\begin{equation}\label{vlog}
V(x)=\mathcal C_1\log |x|+O\left(\frac1{|x|}\right),\qquad\nabla V(x)=\mathcal C_1\frac x{|x|^2}+O\left(\frac1{|x|^2}\right),
\end{equation}
where
$$\mathcal C_1=\int_0^{+\infty}r\frac{r^2-1}{r^2+1}e^{U(r)}f_1(U(r))\mathrm dr=12(1-\log 2),$$
and similarly
\begin{equation}\label{wlog}
W(x)=\mathcal C_2\log |x|+O\left(\frac1{|x|}\right),\qquad\nabla W(x)=\mathcal C_2\frac x{|x|^2}+O\left(\frac1{|x|^2}\right),
\end{equation}
with
$$\mathcal C_2=\int_0^{+\infty}r\frac{r^2-1}{r^2+1}e^{U(r)}f_2(U(r),V(r))\mathrm dr;$$
$\mathcal C_2$ does not have an explicit form, but the integral is finite because, from \eqref{bubble} and \eqref{vlog}, both $U$ and $V$ grow logaritmically at infinity, whereas $f$ grows polynomially and $e^{U(r)}$ decays as $\frac1{r^4}$.\\
The suitable rescalements
$$ 
V_1(x):=V\left(\frac x\alpha\right),\ W_1(x):=W\left(\frac x\alpha\right),\
V_2(x):=V\left(\frac{x^k-\rho^k}{\beta^k}\right),\ W_2(x):=W\left(\frac{x^k-\rho^k}{\beta^k}\right)
 $$
solve, respectively, for $x\in\mathbb R^2$:
\begin{eqnarray*}
&&\Delta V_1(x)+e^{U_1(x)}V_1(x)=e^{U_1(x)}f_1\left(U\left(\frac x\alpha\right)\right);\\
&&\Delta W_1(x)+e^{U_1(x)}W_1(x)=e^{U_1(x)}f_2\left(U\left(\frac x\alpha\right),V\left(\frac x\alpha\right)\right);\\
&&\Delta V_2(x)+|x|^{2k-2}e^{U_2(x)}V_2(x)=|x|^{2k-2}e^{U_2(x)}f_1\left(U\left(\frac{x^k-\rho^k}{\beta^k}\right)\right);\\
&&\Delta W_2(x)+|x|^{2k-2}e^{U_2(x)}W_2(x)=|x|^{2k-2}e^{U_2(x)}f_2\left(U\left(\frac{x^k-\rho^k}{\beta^k}\right),V\left(\frac{x^k-\rho^k}{\beta^k}\right)\right).
\end{eqnarray*}

They will be all used to get extra lower order terms for the approximation of the solution (see Subsection \ref{subsection:ansatzTotal} below).

\subsection{The final ansatz: fixing the Dirichlet boundary conditions}\label{subsection:ansatzTotal}

We will look for $k-$symmetric solutions to \eqref{eqp} (i.e. solutions in the space defined in \eqref{def:H10sim}) in the form
\begin{equation}\label{uphi}
 {u=\Upsilon+\phi},
\end{equation}
where $\Upsilon$ is explicit and given by:
\begin{equation}\label{upsp}
 {\Upsilon:=\underbrace{\tau\left(\mathrm PU_1+\frac{\mathrm PV_1}p+\frac{\mathrm PW_1}{p^2}\right)}_{=:\Upsilon_1}-\underbrace{\tau\eta\left(\mathrm PU_2+\frac{\mathrm PV_2}p+\frac{\mathrm PW_2}{p^2}\right)}_{:=\Upsilon_2}},
\end{equation}
for some $\tau,\eta>0$. Here $U_{1}$ and $U_{2}$ are the functions in Subsection \ref{subsection:ansatzFirst}, while
$V_{1}, V_{2}, W_{1}$ and $W_{2}$ are the ones defined in Subsection 
\ref{subsection:ansatzSecond}, and we are considering their projections on $H^{1}_{0}(\Omega)$. Recall that for any function $\varphi\in H^{1}(\Omega)$ its projection $\mathrm P \varphi$ on $H^{1}_{0}(\Omega)$ is defined as the solution of 
\begin{equation}\label{proie} -\Delta(\mathrm P\varphi)=-\Delta \varphi\ \mbox{in}\ \Omega,\quad P\varphi=0\ \mbox{on}\ \partial\Omega.\end{equation}
The function $\Upsilon$ is clearly $k-$symmetric. 
Moreover, it is important to point out that
 since $\Upsilon_1$ depends on the positive parameters $\alpha, \tau$ and $\Upsilon_2$ depends on $k\in\mathbb N$ and on the positive parameters $\beta, \rho, \tau, \eta$, then $\Upsilon$ depends on five parameters. Whereas $k$ is fixed, $\alpha,\beta,\rho,\tau$ and $\eta$ all depend on $p$. The parameters $\alpha,\beta,\rho,\tau$ will all go to $0$ as $p$ goes to $+\infty$, while the parameter $\eta$ will be close to 
 \begin{equation}\label{eta0}
 \eta_{\infty,k}:=\frac2{k-1}.\end{equation}
 At this stage the choice of $\eta$ could sound baffling but it will be clear in the last step of the finite-dimensional reduction (see Proposition \ref{reduction}).
 All the five parameters shall satisfy the four conditions \eqref{param} in Section \ref{3}, so that only one parameter remains free to be chosen, as we will discuss later on in Section \ref{3} (see in particular \eqref{param} and Lemma \ref{asymptotic}).
 This fact suggests that the higher order term $\phi\in H^1_{0,k}(\Omega)$ (see \eqref{def:H10sim}) in \eqref{uphi} must belong to a suitable space 
 of codimension $1$. Actually, it shall satisfy an orthogonality condition (see \eqref{orto} in Section \ref{4}).
 
\subsection{The reduction process}\label{SubsectionReduction}
 We aim to find a solution to \eqref{eqp} as in \eqref{uphi} where the higher order term $\phi\in H^1_{0,k}(\Omega)$ solves the equation
\begin{equation}\label{rln}
\mathscr E+\mathscr L\phi+\mathscr N(\phi)=0,
\end{equation}
where $\mathscr E$ is the \emph{error term}:
\begin{equation}\label{erre}\mathscr E:=\Delta\Upsilon+|\Upsilon|^{p-1}\Upsilon,\end{equation}
$\mathscr L \phi$ is the \emph{linear term}:
\begin{equation}\label{elle}\mathscr L\phi:=\Delta\phi+p|\Upsilon|^{p-1}\phi\end{equation}
and $\mathscr N(\phi)$ is the \emph{nonlinear term}:
\begin{equation}\label{enne}\mathscr N(\phi):=|\Upsilon+\phi|^{p-1}(\Upsilon+\phi)-|\Upsilon|^{p-1}\Upsilon-p|\Upsilon|^{p-1}\phi.\end{equation}

Throughout all the paper we will assume that $p$ is a \emph{large} parameter, namely we will fix once and for all some $p_0>1$ and we will assume $p>p_0$; the estimates explicitly depending on $p$ are meant for $p>p_0$. Similarly, since $\alpha,\beta,\rho,\tau$ vanish as $p$ grows, the estimates explicitly depending on such parameters are meant for small values of the parameters. We will make no further comments on this.

 \section{The choice of the parameters and the size of the error}\label{3}
We need to choose all the positive parameters $\alpha,\beta,\rho,\tau,\eta$ in the ansatz \eqref{upsp}, which depend on $p$, so that the error term $\mathscr E$ given in \eqref{erre} is small enough as $p\to+\infty.$ 
We keep all the parameters as functions of the parameter $\eta$ (Subsection \ref{subsection:parametersineta}). The free parameter $\eta$ will be derived in the last step of the finite-dimensional reduction (see Proposition \ref{reduction}) and found in a neighbourhood of $\eta_{\infty,k}$ defined in \eqref{eta0}. Hence in Subsection \ref{subsection:preliminaryEta} we collect some preliminary estimates on the parameters $\alpha,\beta,\rho$, assuming that $\eta$ is close enough to $\eta_{\infty,k}$ (see Lemma \ref{alphabetarho}). In Subsection \ref{subsection:ansatzStime} we also collect some preliminary estimates on the profiles in the ansatz in term of the parameters $\alpha,\beta,\rho$.
Finally, exploiting these preliminary results, we estimate the size of the error $\mathscr E$ in terms of $p$, assuming that $\eta$ is close to $\eta_{\infty,k}.$ More precisely, we first estimate the error close to the positive peak (Subsection \ref{Section:errorOrigin}), next close to the negative peaks (see Subsection \ref{Section:errorRoots}), and finally far away from the peaks (Subsection \ref{Section:errorFaraway}). The final estimate of $\mathscr E$ is contained in Proposition \ref{r} (see Subsection \ref{Section:errorTotale}).

 \subsection{The choice of parameters as functions of $\eta$}\label{subsection:parametersineta}
The function $\Upsilon$ in \eqref{upsp} depends on the integer $k\in\mathbb N$, which is fixed, and the parameters $\mathscr E$, which all depend on $p$. In particular, such parameters are choosen by the following relations:
\begin{equation}\label{param}
\left\{\begin{array}{l}\frac\tau{\alpha^2}=(p\tau)^p\\\tau\eta\frac{k^2\rho^{2k-2}}{\beta^{2k}}=(p\tau\eta)^p\\\left(4-\frac{\mathcal C_1}p-\frac{\mathcal C_2}{p^2}\right)\log \frac{\rho^{k\eta}}{\alpha}-\log 8-2(k-1)\pi\left(4-\frac{\mathcal C_1}p-\frac{\mathcal C_2}{p^2}\right)H(0,0)=p
\\
\left(4-\frac{\mathcal C_1}p-\frac{\mathcal C_2}{p^2}\right)\log \frac{\rho^{\frac{1}{\eta}}}{\beta^k}-\log 8+2(k-1)\pi\left(4-\frac{\mathcal C_1}p-\frac{\mathcal C_2}{p^2}\right)H(0,0)=p\end{array}\right.
\end{equation}
where $\mathcal C_1,\mathcal C_2$ are given by \eqref{vlog} and \eqref{wlog}.\\
By means of the implicit function theorem, one easily sees that, for large $p$, the solutions to the $4$ equations in $5$ variables in \eqref{param} are a $1$-parameter family where the dependence of $\alpha,\beta,\rho,\tau$ with respect of $\eta$ can be made explicit. 
 {We agree that we keep all the parameters as functions of the parameter $\eta$ which depends on $p$. In the following Lemma we give their asymptotic values, in dependence of $\eta$ as $p\to+\infty$. It is important to point out that from now on the parameter $\eta$ must be far away from the value 1.
\begin{lemma}\label{asymptotic} 
The parameters $\tau,\alpha,\beta,\rho$ satisfy the following:
$$\left\{\begin{array}{l}\tau=\frac{\sqrt e}p\eta^{-\frac{k\eta^2}{(k\eta-1)(\eta+1)}}+O\left(\frac{\log p}{p^2}\right)\\\rho=e^{\frac{\eta\log \eta}{2(k\eta-1)(\eta+1)}p}\left(K_1(\eta)+O\left(\frac1p\right)\right)\\\alpha=e^{-\frac p4\left(1-\frac{2k\eta^2\log \eta}{(k\eta-1)(\eta+1)}\right)}\left(K_2(\eta)+O\left(\frac1p\right)\right)\\\beta=e^{-\frac p{4k}\left(1-\frac{2\log \eta}{(k\eta-1)(\eta+1)}\right)}\left(K_3(\eta)+O\left(\frac1p\right)\right)\end{array}\right.$$
Here, $K_i(\eta)$ are positive constants, smoothly depending on $\eta$, which are uniformly bounded when $\frac1k<C\le\eta\le C'<1$.\\
The estimates are uniform for $\eta$ in such a range and the same estimates hold true for the derivative in $\eta$ of all parameters.
\end{lemma}

\begin{proof} 
We can write, from \eqref{param},
\begin{equation}\label{system}
\left\{\begin{array}{l}-2\log \alpha=p\log p+(p-1)\log \tau\\(2k-2)\log \rho-2k\log \beta=p\log p+(p-1)(\log \tau+\log \eta)-\log \left(k^2\right)\\4k\eta\log \rho-4\log \alpha=p+\frac{\mathcal C_1}4+8(k-1)\pi H(0,0)+\log 8+O\left(\frac1p\right)\\\frac4\eta\log \rho-4k\log \beta=p+\frac{\mathcal C_1}4-8(k-1)\pi H(0,0)+\log 8+O\left(\frac1p\right)\end{array}\right..
\end{equation}
Subtracting from the third equation twice the first equation, and then from the fourth equation twice the second equation, we get the two equalities:
$$\left\{\begin{array}{l}4k\eta\log \rho=p-2p\log p-2(p-1)\log \tau+\frac{\mathcal C_1}4+8(k-1)\pi H(0,0)\\\qquad+\log 8+O\left(\frac1p\right)\\\left(\frac4\eta-4k+4\right)\log \rho=p-2p\log p-2(p-1)(\log \tau+\log \eta)\\\qquad-8(k-1)\pi H(0,0)+\frac{\mathcal C_1}4+\log 8+4\log k+O\left(\frac1p\right)\end{array}\right.;$$
therefore,
\begin{eqnarray*}
\log \rho&=&\frac{2(p-1)\log \eta+16\pi(k-1)H(0,0)-4\log k+O\left(\frac1p\right)}{4k\eta-\frac4\eta +4k -4}\\
&=&(p-1)\frac{\eta\log \eta}{2(k\eta-1)(\eta+1)}+\frac{4(k-1)\eta\pi H(0,0)-\eta\log k}{(k\eta-1)(\eta+1)}+O\left(\frac1p\right),
\end{eqnarray*}
which gives the estimate for $\rho$.
Then, from \eqref{system}, we get
\begin{eqnarray*}
\alpha&=&\rho^{k\eta}e^{-\frac p4 -\frac{C_1}{16}-2(k-1)\pi H(0,0)-\frac{\log 8}4+O\left(\frac1p\right)}\\
&=&e^{\frac{k\eta^2\log \eta}{2(k\eta-1)(\eta+1)}(p-1)}e^{-\frac p4}e^{-\frac{C_1}{16}-2(k-1)\pi H(0,0)-\frac{\log 8}4}\left(K_1(\eta)+O\left(\frac1p\right)\right)\left(1+O\left(\frac1p\right)\right)\\
&=&e^{-\frac p4\left(1-\frac{2k\eta^2\log \eta}{(k\eta-1)(\eta+1)}\right)}\left(K_2(\eta)+O\left(\frac1p\right)\right),
\end{eqnarray*}
\begin{eqnarray*}
\beta&=&\rho^\frac1{k\eta}e^{-\frac p{4k}-\frac{C_1}{16}+2\frac{k-1}k\pi H(0,0)-\frac{\log 8}{4k}+O\left(\frac1p\right)}\\
&=&e^{\frac{\log \eta}{2k(k\eta-1)(\eta+1)}(p-1)}e^{-\frac p{4k}}e^{-\frac{C_1}{16}+2\frac{k-1}k\pi H(0,0)-\frac{\log 8}{4k}}\left(K_1(\eta)+O\left(\frac1p\right)\right)\left(1+O\left(\frac1p\right)\right)\\
&=&e^{-\frac p{4k}\left(1-\frac{2\log \eta}{(k\eta-1)(\eta+1)}\right)}\left(K_3(\eta)+O\left(\frac1p\right)\right),
\end{eqnarray*}
and
\begin{eqnarray*}
\tau&=&p^{-\frac p{p-1}}\alpha^{-\frac2{p-1}}\\
&=&p^{-\frac p{p-1}}e^{\frac p{2(p-1)}\left(1-\frac{2k\eta^2\log \eta}{(k\eta-1)(\eta+1)}\right)}\left(1+O\left(\frac1p\right)\right)\\
&=&\frac{\sqrt e}p\eta^{-\frac{k\eta^2}{(k\eta-1)(\eta+1)}}+O\left(\frac{\log p}{p^2}\right).
\end{eqnarray*}
The uniformity of the limits with respect to $\eta$ is straightforward. The convergence of the derivatives follows from the smoothness of equations in \eqref{param} and the implicit function theorem.
\end{proof}

\subsection{Preliminaries estimates when $\eta$ is close to $\eta_{\infty,k}$}\label{subsection:preliminaryEta}

We choose the parameter $\eta$ in a neighbourhood of $\eta_{\infty,k}$ given in \eqref{eta0}
 and we prove a couple of technical but crucial estimates.
 It is important to point out that if $k=3$ then $\eta_{\infty,3}=1$ and Lemma \ref{asymptotic} cannot be applied.\\

\begin{lemma}\label{condeta} 
If $k=4,5$, then there exists $\epsilon>0$ such that if 
\begin{equation}\label{eta}
\eta_{\infty,k}-\epsilon\le\eta\le\eta_{\infty,k}+\epsilon,
\end{equation}
where $\eta_{k}$ is defined in \eqref{eta0}, then
\begin{eqnarray}
\label{eta1} &\frac2k<\eta<1,\\
\label{etacond}& \frac12+\frac{\log \eta}{\eta+1}>0,\\
\label{etaalpha} &\frac12+\frac{k\eta^2}{(k\eta-1)(\eta+1)}\log \eta>0,\\
\label{etak} &\frac12+\frac{(k+1)\eta}{2(k\eta-1)(\eta+1)}\log \eta>0.
\end{eqnarray}
\end{lemma}

\begin{proof} 
One clearly suffices to show that all inequalities hold for $\eta=\eta_{\infty,k}$ and then argue by continuity; \eqref{eta1} is immediate to be verified. \eqref{etacond} holds true for $\eta_{\infty,k}$ because
$$\frac12+\frac{\log \eta_{\infty,k}}{\eta_{\infty,k}+1}=\frac12-\frac{k-1}{k+1}\log \frac{k-1}2>0,$$
since $k<k_0$, $k_0$ solves \eqref{kzero} and $k\mapsto\frac12-\frac{k-1}{k+1}\log \frac{k-1}2$ strictly decreases; \eqref{etak} also holds true because
$$\frac12+\frac{(k+1)\eta_{\infty,k}}{2(k\eta_{\infty,k}-1)(\eta_{\infty,k}+1)}\log \eta_{\infty,k}=\frac12+\frac{\log \eta_{\infty,k}}{\eta_{\infty,k}+1}>0.$$
Finally, $\eta_{\infty,k}$ satisfies \eqref{etaalpha} because, being $k\mapsto\frac{4k}{(k+1)^2}\log \frac2{k-1}$ decreasing, then
\begin{eqnarray*}
\frac12+\frac{k\eta_{\infty,k}^2}{(k\eta_{\infty,k}-1)(\eta_{\infty,k}+1)}\log \eta_{\infty,k}&=&\frac12+\frac{4k}{(k+1)^2}\log \frac2{k-1}\\
&\ge&\frac12+\frac{4k_0}{(k_0+1)^2}\log \frac2{k_0-1}\\
&>&\frac12+\frac{k_0-1}{k_0+1}\log \frac2{k_0-1}\\
&=&0.
\end{eqnarray*}
\end{proof}

Under this choice, not only the assumptions on $\eta$ of Lemma \ref{asymptotic} are satisfied but the concentration rates $\alpha,\beta$ of the \emph{positive} and \emph{negative} bubbles are both faster than the rate $\rho$ at which bubbling are approaching each other. 

\begin{lemma}\label{alphabetarho} 
Let $k=4,5$, and let $\eta$ satisfy the assumption \eqref{eta} in Lemma \ref{condeta}.
Then there exists $\varepsilon>0$ such that, for $p$ large enough,
\begin{eqnarray}
\label{alpharho}\frac\alpha\rho&\le&e^{-\varepsilon p},\\
\label{betarho}\frac\beta\rho&\le&e^{-\varepsilon p},\\
\label{alpharhok}\frac\alpha{\rho^{2k\eta}}&\le&e^{-\varepsilon p}.
\end{eqnarray}
Moreover, there exists $\theta_0>0$ such that, if $0<\theta\le\theta_0$, then
\begin{eqnarray}
\label{alphatheta}\frac{\alpha^{1-\theta}}\rho&\le&e^{-\frac\varepsilon2p},\\
\label{betatheta}\frac{\beta^{1-\theta}}\rho&\le&e^{-\frac\varepsilon2p}.
\end{eqnarray}
\end{lemma}

\begin{proof} 
From Lemma \ref{asymptotic} we get
$$\frac\alpha\rho\le C\rho^{k\eta-1}e^{-\frac p4}\le C\rho^{k(\eta_{\infty,k}-\epsilon)-1}e^{-\frac p4}\le e^{-\frac p4}$$
for large $p$, since $k(\eta_{\infty,k}-\epsilon)-1>0$ by \eqref{eta1}, which gets \eqref{alpharho}.\\
Using again Lemma \ref{asymptotic} we get
$$\frac\beta\rho\le C\frac{e^{-\frac p{4k}}}{\rho^\frac{k\eta-1}{k\eta}}\le Ce^{-\frac p{2k}\left(\frac12+\frac{\log \eta}{\eta+1}\right)}\le Ce^{-\frac1{2k}\left(\frac12+\frac{\log (\eta_{\infty,k}+\epsilon)}{\eta_{\infty,k}+\epsilon+1}\right)p}\le e^{-\frac1{4k}\left(\frac12+\frac{\log (\eta_{\infty,k}+\epsilon)}{\eta_{\infty,k}+\epsilon+1}\right)p},$$
since $\frac12+\frac{\log (\eta_{\infty,k}+\epsilon)}{\eta_{\infty,k}+\epsilon+1}>0$ from \eqref{etacond}.\\
Similarly, in view of \eqref{etaalpha},
\begin{eqnarray*}
\frac\alpha{\rho^{2k\eta}}&\le&Ce^{-\frac p2\left(\frac12+\frac{k\eta^2}{(k\eta-1)(\eta+1)}\right)\log \eta}\\
&\le&Ce^{-\frac p2\left(\frac12+\frac{k(\eta_{\infty,k}+\epsilon)^2}{(k(\eta_{\infty,k}+\epsilon)-1)(\eta_{\infty,k}+\epsilon+1)}\log (\eta_{\infty,k}+\epsilon)\right)}\\
&\le&e^{-\frac p4\left(\frac12+\frac{k(\eta_{\infty,k}+\epsilon)^2}{(k(\eta_{\infty,k}+\epsilon)-1)(\eta_{\infty,k}+\epsilon+1)}\log (\eta_{\infty,k}+\epsilon)\right)},
\end{eqnarray*}
therefore, to prove the first part of the Lemma, we suffice to choose
$$\varepsilon:=\frac14\min\left\{1,\frac1k\left(\frac12+\frac{\log (\eta_{\infty,k}+\epsilon)}{\eta_{\infty,k}+\epsilon+1}\right),\frac12+\frac{k(\eta_{\infty,k}+\epsilon)^2}{(k(\eta_{\infty,k}+\epsilon)-1)(\eta_{\infty,k}+\epsilon+1)}\log (\eta_{\infty,k}+\epsilon)\right\}.$$

For the second part, we first observe that, since $\alpha,\beta<1$, both \eqref{alphatheta} and \eqref{betatheta} just suffice to be shown in the case $\theta=\theta_0$.\\
From Lemma \ref{asymptotic} we get
$$\alpha^{\theta_0}\ge\frac1Ce^{-\frac p4\left(1-\frac{2k\eta^2\log \eta}{(k\eta-1)(\eta+1)}\right)\theta_0}\ge\frac1Ce^{-p\left(1-\frac{2k(\eta_{\infty,k}-\epsilon)^2\log (\eta_{\infty,k}-\epsilon)}{(k(\eta_{\infty,k}-\epsilon)-1)(\eta_{\infty,k}+\epsilon+1)}\right)\frac{\theta_0}4}.$$
Therefore, choosing $\theta_0$ so small that
$$\left(1-\frac{2k(\eta_{\infty,k}-\epsilon)^2\log (\eta_{\infty,k}-\epsilon)}{(k(\eta_{\infty,k}-\epsilon)-1)(\eta_{\infty,k}+\epsilon+1)}\right)\frac{\theta_0}4\le\frac\varepsilon4$$
we get
$$\frac{\alpha^{1-\theta_0}}\beta\le Ce^{-\frac34\varepsilon p}\le e^{-\frac\varepsilon2p}.$$
\eqref{betatheta} follows similarly from \eqref{betarho} and Lemma \ref{asymptotic}.
\end{proof}

 From now on we agree that $\eta$ is chosen as in \eqref{eta} so that Lemma \ref{alphabetarho} holds true.

\subsection{Preliminaries estimates for the ansatz}\label{subsection:ansatzStime} We introduce some estimates which will be crucial in the computation of the error size.
For $x,y\in\Omega$ we denote as $G(x,y)$ the Green's function for $-\Delta$ on $\Omega$ and as $H(x,y)$ its regular part, that is:
$$\left\{\begin{array}{ll}-\Delta_xG(x,y)=\delta_y(x)&x\in\Omega\\G(x,y)=0&x\in\partial\Omega\end{array}\right.,\qquad H(x,y):=G(x,y)+\frac1{2\pi}\log |x-y|;$$
we point out that $H$ solves
\begin{equation}\label{h}
\left\{\begin{array}{ll}-\Delta_xH(x,y)=0&x\in\Omega\\H(x,y)=\frac1{2\pi}\log |x-y|&x\in\partial\Omega\end{array}\right..
\end{equation}

\

The following estimates will be very convenient in several parts of this work:
\begin{lemma}\label{pupvpw} Let $U_{1}$ and $U_{2}$ be the profiles introduced in Subsection \ref{subsection:ansatzFirst} and $V_{1}, V_{2}, W_{1}$ and $W_{2}$ be the ones defined in Subsection 
\ref{subsection:ansatzSecond}. 
The following estimates for their projections on $H^1_0(\Omega)$ (see \eqref{proie} for the definition) hold uniformly in $x\in\Omega$:
\begin{eqnarray}
\label{pu1}\mathrm PU_1(x)&=&U\left(\frac x\alpha\right)+8\pi H(0,0)-\log 8-4\log \alpha+O\left(|x|+\alpha^2\right),\\
\nonumber\mathrm PV_1(x)&=&V\left(\frac x\alpha\right)-2\mathcal C_1\pi H(0,0)-\mathcal C_1\log \alpha+O(|x|+\alpha),\\
\nonumber\mathrm PW_1(x)&=&W\left(\frac x\alpha\right)-2\mathcal C_2\pi H(0,0)-\mathcal C_2\log \alpha+O(|x|+\alpha),\\
\label{pu2out}\mathrm PU_2(x)&=&U\left(\frac{x^k-\rho^k}{\beta^k}\right)+8k\pi H(0,0)-\log \left(8k^2\right)-4k\log \beta\\
\nonumber&+&O\left(\left|x^k-\rho^k\right|+\beta^{2k}+\rho^k\right),\\
\nonumber\mathrm PV_2(x)&=&V\left(\frac{x^k-\rho^k}{\beta^k}\right)-2k\mathcal C_1\pi H(0,0)-k\mathcal C_1\log \beta\\
\nonumber&+&O\left(\left|x^k-\rho^k\right|+\beta^k+\rho^k\right),\\
\nonumber\mathrm PW_2(x)&=&W\left(\frac{x^k-\rho^k}{\beta^k}\right)-2k\mathcal C_2\pi H(0,0)-k\mathcal C_2\log \beta\\
\nonumber&+&O\left(\left|x^k-\rho^k\right|+\beta^k+\rho^k\right),
\end{eqnarray}
where $U,V,W$ are the functions defined in \eqref{bubble} and \eqref{vw}.
Moreover, if $|y|\le\frac12\frac\rho\alpha$, then:
\begin{eqnarray}
\label{pu2}\mathrm PU_2(\alpha y)&=&-4k\log \rho+8k\pi H(0,0)+O\left(\frac{\alpha^k}{\rho^k}|y|^k+\frac{\beta^{2k}}{\rho^{2k}}+\rho^k\right);\\
\nonumber\mathrm PV_2(\alpha y)&=&k\mathcal C_1\log \rho-2k\mathcal C_1\pi H(0,0)+O\left(\frac{\alpha^k}{\rho^k}|y|^k+\frac{\beta^k}{\rho^k}+\rho^k\right);\\
\nonumber\mathrm PW_2(\alpha y)&=&k\mathcal C_2\log \rho-2k\mathcal C_2\pi H(0,0)+O\left(\frac{\alpha^k}{\rho^k}|y|^k+\frac{\beta^k}{\rho^k}+\rho^k\right).
\end{eqnarray}
If $|y|\le\frac12\frac{\rho^k}{\beta^k}$, then:
\begin{eqnarray}
\label{pu1out}\mathrm PU_1\left(\sqrt[k]{\beta^ky+\rho^k}\right)&=&-4\log \rho+8\pi H(0,0)+O\left(\frac{\beta^k}{\rho^k}|y|+\frac{\alpha^2}{\rho^2}+\rho\right);\\
\nonumber\mathrm PV_1\left(\sqrt[k]{\beta^ky+\rho^k}\right)&=&\mathcal C_1\log \rho-2\mathcal C_1\pi H(0,0)+O\left(\frac{\beta^k}{\rho^k}|y|+\frac\alpha\rho+\rho\right);\\
\nonumber\mathrm PW_1\left(\sqrt[k]{\beta^ky+\rho^k}\right)&=&\mathcal C_2\log \rho-2\mathcal C_2\pi H(0,0)+O\left(\frac{\beta^k}{\rho^k}|y|+\frac\alpha\rho+\rho\right).
\end{eqnarray}
\end{lemma}

\begin{proof} 
We follow \cite{egp}, Proposition A.1 (see also \cite[Section 2]{EspositoMussoPistoiaJDE2006}): since the harmonic function $\mathrm PU_1-U_1$ satisfies, for $x\in\partial\Omega$,
$$
\mathrm PU_1(x)-U_1(x)=-U_1(x)=4\log |x|-\log \left(8\alpha^2\right)+O\left(\alpha^2\right),
$$
then, recalling that $H(x,y)$ solves \eqref{h}, the maximum principle gets:
\begin{eqnarray*}
\mathrm PU_1(x)&=&U_1(x)+8\pi H(x,0)-\log \left(8\alpha^2\right)+O\left(\alpha^2\right)\\
&=&U\left(\frac x\alpha\right)+8\pi H(0,0)-\log 8-4\log \alpha+O\left(\alpha^2+|x|\right),
\end{eqnarray*}
namely \eqref{pu1}. Then, from \eqref{bubble} we deduce
\begin{equation}\label{ulog}
U(x)=-4\log |x|+\log 8+O\left(\frac1{|x|^2}\right);
\end{equation}
combining with \eqref{pu1}, one has
$$\mathrm PU_1(x)=-4\log |x|+8\pi H(0,0)+O\left(|x|+\frac{\alpha^2}{|x|^2}\right)$$
which, evaluated in $x=\sqrt[k]{\beta^ky+\rho^k}=\rho\left(1+O\left(\frac{\beta^k}{\rho^k}|y|\right)\right)$,
gives \eqref{pu1out}.\\
Applying the same argument to $U_2$ gives
\begin{eqnarray}
\nonumber\mathrm PU_2(x)&=&U_2(x)+8\pi H\left(x^k,\rho^k\right)-\log \left(8k^2\beta^{2k}\right)+O\left(\beta^{2k}\right)\\
\nonumber&=&U\left(\frac{x^k-\rho^k}{\beta^k}\right)+8k\pi H(0,0)-\log \left(8k^2\right)-4k\log \rho\\
\nonumber&+&O\left(\beta^{2k}+\left|x^k-\rho^k\right|+\rho^k\right)\\
\label{u2}&=&-4\log \left|x^k-\rho^k\right|+8k\pi H(0,0)-\log \left(8k^2\right)\\
\nonumber&+&O\left(\left|x^k-\rho^k\right|+\rho^k+\frac{\beta^{2k}}{|x^k-\rho^k|^2}\right)
\end{eqnarray}
since $H\left(x^k,0\right)=kH(x,0)$ and \eqref{ulog} holds. \\
Indeed,  $H(x,0)$ and  
$H(x^k,0)$ are harmonic functions which (in complex notation) satisfy
$$H(x^k,0)=\frac k{2\pi}\ln|x|=kH(x,0) \ \hbox{on}\ \partial\Omega.$$
Thus, we got \eqref{pu2} and, evaluating \eqref{u2} in $x=\alpha y$, \eqref{pu2out}.\\
The estimates for $V,W$ are shown similarly, in view of the asymptotic behaviors \eqref{vlog}, \eqref{wlog}.
\end{proof}
 
\subsection{From the large power $u^p$ to the exponential $e^u$}\label{upversuseu}
 The following Taylor expansion has been used in the construction of the ansatz. It is also a key tool to estimate the size of the error.
\begin{lemma}\label{taylor} 
For any real numbers $a,b,c$ such that for some $C>0$ one has
\begin{equation}\label{abc}
-\left(1-\frac1C\right)p\le a\le C,\qquad|b|+|c|\le Cp,
\end{equation}
the following asymptotic expansions hold true, as $p$ goes to $+\infty$:
\begin{eqnarray}
\label{taylorp}&&\left(1+\frac ap+\frac b{p^2}+\frac c{p^3}\right)^p\\
\nonumber&=&e^a\left(1+\frac1p\left(b-\frac{a^2}2\right)+\frac1{p^2}\left(c-ab+\frac{a^3}3+\frac12\left(b-\frac{a^2}2\right)^2\right)\right)\\
\nonumber&+&O\left(e^a\frac{1+a^6+b^6+c^6}{p^3}\right).
\end{eqnarray}
\end{lemma}
\begin{proof} 
By writing
$$\log (1+x)=x-\frac{x^2}2+\frac{x^3}3-\int_0^x\frac{(x-t)^3}{(1+t)^4}\mathrm dt,$$
with $x=\frac ap+\frac b{p^2}+\frac c{p^3}$, we have
\begin{eqnarray*}
-\frac{x^2}2&=&-\frac{a^2}{2p^2}-\frac{ab}{p^3}+O\left(\frac{a^2+b^2+c^2}{p^4}\right),\\
\frac{x^3}3&=&\frac{a^3}{3p^3}+O\left(\frac{|a|^3+|b|^3+|c|^3}{p^4}\right);
\end{eqnarray*}
Moreover, since \eqref{abc} yields $\frac1C\le1+x\le C$, then
$$0\le\int_0^x\frac{(x-t)^3}{(1+t)^4}\mathrm dt=\int_0^xO\left(|x-t|^3\right)\mathrm dt=O\left(x^4\right)=O\left(\frac{a^4+b^4+c^4}{p^4}\right),$$
therefore,
\begin{eqnarray}
\nonumber y&:=&p\log \left(1+\frac ap+\frac b{p^2}+\frac c{p^3}\right)-a\\
\nonumber&=&\frac1p\left(b-\frac{a^2}2\right)+\frac1{p^2}\left(c-ab+\frac{a^3}3\right)+O\left(\frac{1+|a|^3+|b|^3+|c|^3}{p^3}\right)\\
\label{bigo}&+&O_{\le0}\left(\frac{a^4+b^4+c^4}{p^3}\right).
\end{eqnarray}
Now, we write
$$e^y=1+y+\frac{y^2}2+\frac12\int_0^y(y-t)^2e^t\mathrm dt,$$
with
$$\frac{y^2}2=\frac1{2p^2}\left(b-\frac{a^2}2\right)^2+O\left(\frac{1+|a|^5+|b|^3+|c|^3}{p^3}\right);$$
from \eqref{abc} and the fact that the remainder in \eqref{bigo} is negative we get that $y$ is bounded from above, therefore
$$\int_0^y(y-t)^2e^t\mathrm dt=O\left(|y|^3\right)=O\left(\frac{1+a^6+b^6+c^6}{p^3}\right),$$
hence
\begin{eqnarray*}
&&\left(1+\frac ap+\frac b{p^2}+\frac c{p^3}\right)^p\\
&=&e^ae^y\\
&=&e^a\left(1+\frac1p\left(b-\frac{a^2}2\right)+\frac1{p^2}\left(c-ab+\frac{a^3}3+\frac12\left(b-\frac{a^2}2\right)^2\right)\right.\\
&+&\left.O\left(\frac{1+a^6+b^6+c^6}{p^3}\right)\right)
\end{eqnarray*}
\end{proof}

\subsection{The error close to the positive peak (i.e. the origin)}\label{Section:errorOrigin}
\begin{lemma}\label{in1} 
There exist $\theta\in(0,1),C>0$ such that
$$|y|\le\frac1{\alpha^\theta}\qquad\Rightarrow\qquad|\mathscr E(\alpha y)|\le\frac C{\alpha^2p^4}\frac{\log ^6(|y|+2)}{\left(|y|^2+1\right)^2}$$
\end{lemma}

\begin{proof} 
From Lemma \ref{pupvpw} and the fact that $|\alpha y|\le\alpha^{1-\theta}$ we get:
\begin{eqnarray*}
\Upsilon(\alpha y)&=&\tau\left(\mathrm PU_1(\alpha y)+\frac{\mathrm PV_1(\alpha y)}p+\frac{\mathrm PW_1(\alpha y)}{p^2}\right)\\
&-&\tau\eta\left(\mathrm PU_2(\alpha y)+\frac{\mathrm PV_2(\alpha y)}p+\frac{\mathrm PW_2(\alpha y)}{p^2}\right)\\
&=&\tau\left(U(y)+\frac{V(y)}p+\frac{W(y)}{p^2}+2\pi\left(4-\frac{\mathcal C_1}p-\frac{\mathcal C_2}{p^2}\right)H(0,0)-\log 8\right.\\
&&\left.+\left(-4+\frac{\mathcal C_1}p+\frac{\mathcal C_2}{p^2}\right)\log \alpha+O\left(\alpha^{1-\theta}\right)\right)\\
&-&\tau\eta\left(\left(-4+\frac{\mathcal C_1}p+\frac{\mathcal C_2}{p^2}\right)k\log \rho+2k\pi\left(4-\frac{\mathcal C_1}p-\frac{\mathcal C_2}{p^2}\right)H(0,0)\right.\\
&&\left.+O\left(\frac{\alpha^{k(1-\theta)}}{\rho^k}+\frac{\beta^k}{\rho^k}+\rho^k\right)\right).
\end{eqnarray*}
By Lemmas \ref{asymptotic} and \ref{alphabetarho}, the remainder decays exponentially fast with respect to $p$ if $\theta$ is chosen small enough; therefore, using the relations \eqref{param} between the parameters we get:
\begin{equation}\label{estimuvw}
\Upsilon(\alpha y)= \tau p\left(1+\frac{U(y)}{p}+\frac{V(y)}{p^2}+\frac{W(y)}{p^3}+O\left(e^{-\varepsilon p}\right)\right),
\end{equation}
with $\varepsilon$ possibly smaller than in Lemma \ref{alphabetarho}.\\
Now, in view of the explicit expression \eqref{bubble} for $U$ and the logarithmic behaviors \eqref{vlog}, \eqref{wlog} of $V,W$, we are in position to apply Lemma \ref{taylor} with $a=U(y),b=V(y),c=W(y)$ and $|y|\le e^{\left(1-\frac1C\right)\frac p4}$, which holds true if $|y|\le\frac1{\alpha^\theta}$, after taking a possibly smaller range of $\theta$:
\begin{eqnarray}
\nonumber&&\left(1+\frac{U(y)}p+\frac{V(y)}{p^2}+\frac{W(y)}{p^3}\right)^p\\
\nonumber&=&e^{U(y)}\left(1+\frac1p\left(V(y)-\frac{U(y)^2}2\right)\right.\\
\nonumber&&\left.+\frac1{p^2}\left(W(y)-U(y)V(y)+\frac{U(y)^3}3+\frac12\left(V(y)-\frac{U(y)^2}2\right)^2\right)\right)\\
\label{uvw}&+&O\left(\frac1{\left(|y|^2+1\right)^2}\frac{\log ^6(|y|+2)}{p^3}\right).
\end{eqnarray}
Therefore, by the choice of parameters \eqref{param} we get the estimate
\begin{eqnarray*}
&&\Delta\Upsilon_1(\alpha y)+|\Upsilon(\alpha y)|^{p-1}\Upsilon(\alpha y)\\
&=&\frac\tau{\alpha^2}\left(\Delta U(y)+\frac1p\Delta V(y)+\frac1{p^2}\Delta W(y)\right)\\
&+&(\tau p)^p\left(1+\frac{U(y)}{p}+\frac{V(y)}{p^2}+\frac{W(y)}{p^3}+O\left(e^{-\varepsilon p}\right)\right)^p\\
&=&\frac\tau{\alpha^2}\left(-e^{U(y)}\left(1+\frac1p\left(V(y)-\frac{U(y)^2}2\right)\right.\right.\\
&&\left.\left.+\frac1{p^2}\left(W(y)-U(y)V(y)+\frac{U(y)^3}3+\frac12\left(V(y)-\frac{U(y)^2}2\right)^2\right)\right)\right.\\
&&+\left.\left(1+\frac{U(y)}p+\frac{V(y)}{p^2}+\frac{W(y)}{p^3}+O\left(\frac{e^{-\varepsilon p}}p\right)\right)^p\right)\\
&=&O\left(\frac\tau{\alpha^2p^3}\frac{\log ^6(|y|+2)}{\left(|y|^2+1\right)^2}\right).
\end{eqnarray*}
To deal with the other term defining $\mathscr E$, we observe that $U,V,W$ grow logarithmically due to \eqref{bubble}, \eqref{vlog}, \eqref{wlog} whereas $\frac\alpha{\rho^{2k\eta}},\frac\beta\rho$ vanish due to Lemma \ref{alphabetarho}; therefore,
\begin{eqnarray*}
\Delta\Upsilon_2(\alpha y)&=&O\left(\frac{\tau\alpha^{2k-2}|y|^{2k-2}}{\beta^{2k}}e^U\left(1+U^4+V^2+|W|\right)\left(\frac{\alpha^ky^k-\rho^k}{\beta^k}\right)\right)\\
&=&O\left(\frac{\alpha^{(1-\theta)(2k-2)}}{\beta^{2k}}\frac{\log ^{4k}\frac\rho\beta}{\left(\frac\rho\beta\right)^{4k}}\right)\\
&=&O\left(\frac{\rho^{(1-\theta)(2k-2)k\eta}}{\rho^{2k}}\right)\\
&=&O\left(e^{-\varepsilon p}\right),
\end{eqnarray*}
after possibly taking a smaller $\theta$. From this, we can conclude
\begin{eqnarray*}
\mathscr E(\alpha y)&=&\Delta\Upsilon_1(\alpha y)+\Upsilon(\alpha y)^p-\Delta\Upsilon_2(\alpha y)\\
&=&O\left(\frac\tau{\alpha^2p^3}\frac{\log ^6(|y|+2)}{\left(|y|^2+1\right)^2}+e^{-\varepsilon p}\right)\\
&=&O\left(\frac1{\alpha^2p^4}\frac{\log ^6(|y|+2)}{\left(|y|^2+1\right)^2}\right).
\end{eqnarray*}
\end{proof}

\subsection{The error close to the negative peaks (i.e. the roots of $\rho^k$)}\label{Section:errorRoots}

\begin{lemma}\label{in2} 
There exist $\theta\in(0,1),C>0$ such that
$$|y|\le\frac1{\beta^{k\theta}}\qquad\Rightarrow\qquad\left|\mathscr E\left(\sqrt[k]{\beta^ky+\rho^k}\right)\right|\le C\frac{\rho^{2k-2}}{\beta^{2k}p^4}\frac{\log ^6(|y|+2)}{\left(|y|^2+1\right)^2}$$
\end{lemma}

\begin{proof} 
The argument is similar to Lemma \ref{in1}.\\
From Lemma \ref{pupvpw}, \eqref{param} and the exponential decays given by Lemmas \ref{asymptotic}, \ref{alphabetarho} we get, for small $\theta$:
\begin{eqnarray*}
&&\Upsilon\left(\sqrt[k]{\beta^ky+\rho^k}\right)\\
&=&\tau\left(\mathrm PU_1\left(\sqrt[k]{\beta^ky+\rho^k}\right)+\frac{\mathrm PV_1\left(\sqrt[k]{\beta^ky+\rho^k}\right)}p+\frac{\mathrm PW_1\left(\sqrt[k]{\beta^ky+\rho^k}\right)}{p^2}\right)\\
&-&\tau\eta\left(\mathrm PU_2\left(\sqrt[k]{\beta^ky+\rho^k}\right)+\frac{\mathrm PV_2\left(\sqrt[k]{\beta^ky+\rho^k}\right)}p+\frac{\mathrm PW_2\left(\sqrt[k]{\beta^ky+\rho^k}\right)}{p^2}\right)\\
&=&\tau\left(\left(-4+\frac{\mathcal C_1}p+\frac{\mathcal C_2}{p^2}\right)\log \rho+2\pi\left(4-\frac{\mathcal C_1}p-\frac{\mathcal C_2}{p^2}\right)H(0,0)\right.\\
&&\left.+O\left(\frac{\beta^{k(1-\theta)}}{\rho^k}+\frac\alpha\rho+\rho^k\right)\right)\\
&-&\tau\eta\left(U(y)+\frac{V(y)}p+\frac{W(y)}{p^2}+2k\pi\left(4-\frac{\mathcal C_1}p-\frac{\mathcal C_2}{p^2}\right)H(0,0)-\log 8\right.\\
&&\left.+\left(-4+\frac{\mathcal C_1}p+\frac{\mathcal C_2}{p^2}\right)k\log \beta+O\left(\beta^{k(1-\theta)}+\rho^k\right)\right)\\
&=&p\tau\eta\left(1+\frac{U(y)}{p}+\frac{V(y)}{p^2}+\frac{W(y)}{p^3}+O\left(e^{-\frac\varepsilon2p}\right)\right).
\end{eqnarray*}
Since $|y|\le\frac1{\beta^{k\theta}}$, if $\theta$ is small enough we are able to apply Lemma \ref{taylor} to get again \eqref{uvw} and we get:
\begin{eqnarray*}
&&\Delta\Upsilon_2\left(\sqrt[k]{\beta^ky+\rho^k}\right)+\left|\Upsilon\left(\sqrt[k]{\beta^ky+\rho^k}\right)\right|^{p-1}\Upsilon\left(\sqrt[k]{\beta^ky+\rho^k}\right)\\
&=&\frac{\tau k^2\left|\sqrt[k]{\beta^ky+\rho^k}\right|^{2k-2}}{\beta^{2k}}\left(\Delta U(y)+\frac{\Delta V(y)}p+\frac{\Delta W(y)}{p^2}\right)\\
&+&(p\tau)^p\eta^p\left(1+\frac{U(y)}{p}+\frac{V(y)}{p^2}+\frac{W(y)}{p^3}+O\left(e^{-\varepsilon p}\right)\right)^p\\
&=&O\left(\frac{\tau k^2\rho^k}{\beta^{2k}}e^{U(y)}\left(U(y)^4+V(y)^2+|W(y)|\right)\right)\\
&+&O\left(\frac{\tau\rho^{2k-2}}{\beta^{2k}p^3}\frac{\log ^6(|y|+2)}{\left(|y|^2+1\right)^2}+e^{-\varepsilon p}\right)\\
&=&O\left(\frac{\rho^{2k-2}}{\beta^{2k}p^4}\frac{\log ^6(|y|+2)}{\left(|y|^2+1\right)^2}+e^{-\varepsilon p}\right);
\end{eqnarray*}
finally, from Lemmas \ref{asymptotic}, \ref{condeta} and \ref{alphabetarho} we can take $\sigma>0$ such that $\frac{\alpha^\sigma}{\rho^{4(k\eta-1)+\sigma}}\ge1$, therefore from the logarithmic growth of $U,V,W$ we get
\begin{eqnarray*}
\Delta\Upsilon_1\left(\sqrt[k]{\beta^ky+\rho^k}\right)&=&O\left(\frac\tau{\alpha^2}e^U\left(1+U^4+V^2+|W|\right)\left(\frac{\sqrt[k]{\beta^ky+\rho^k}}\alpha\right)\right)\\
&=&O\left(\frac1{\alpha^2}\frac{\log ^4\frac\rho\alpha}{\left(\frac\rho\alpha\right)^4}\right)\\
&=&O\left(\frac{\alpha^{2-\sigma}}{\rho^{4-\sigma}}\right)\\
&=&O\left(\frac{\alpha^2}{\rho^{4k\eta}}\right)\\
&=&O\left(e^{-2\varepsilon p}\right),
\end{eqnarray*}
which concludes the proof.
\end{proof}

\subsection{The error far away from the peaks}\label{Section:errorFaraway}
\begin{lemma}\label{out} 
There exist $\delta\in(0,1),C>0$ such that
$$|x|>\alpha,\,\left|x^k-\rho^k\right|>\beta^k,\qquad\Rightarrow\qquad|\mathscr E(x)|\le\frac Cp\left(\frac{\alpha^{4\delta}}{|x|^{2+4\delta}}+\frac{\beta^{4k\delta}|x|^{2k-2}}{\left|x^k-\rho^k\right|^{2+4\delta}}\right).$$
\end{lemma}
\begin{proof} 
We write $\mathscr E=\Delta\Upsilon_1-\Delta\Upsilon_2+|\Upsilon|^{p-1}\Upsilon$ and we estimate separately each term.\\
In view of the logaritmic behavior of $U,V,W$, we have
\begin{eqnarray}
\nonumber\Delta\Upsilon_1(x)&=&O\left(\frac\tau{\alpha^2}e^U\left(1+U^4+V^2+|W|\right)\left(\frac x\alpha\right)\right)\\
\nonumber&=&O\left(\frac\tau{\alpha^2}\frac{\log ^4\left|\frac x\alpha\right|}{\left|\frac x\alpha\right|^4}\right)\\
\label{out1}&=&O\left(\frac1p\frac\alpha{|x|^3}\right)
\end{eqnarray}
and, similarly,
\begin{eqnarray}
\nonumber\Delta\Upsilon_2(x)&=&O\left(\frac\tau{\alpha^2}e^U\left(1+U^4+V^2+|W|\right)\left(\frac x\alpha\right)\right)\\
\nonumber&=&O\left(\frac{\tau|x|^{2k-2}}{\beta^{2k}}\frac{\log ^4\left|\frac {x^k-\rho^k}{\beta^k}\right|}{\left|\frac {x^k-\rho^k}{\beta^k}\right|^4}\right)\\
\label{out2}&=&O\left(\frac1p\frac{\beta^k|x|^{2k-2}}{\left|x^k-\rho^k\right|^3}\right).
\end{eqnarray}
As for the nonlinear term, in view of the relations \eqref{param} between the parameters we get
\begin{eqnarray*}
\Upsilon(x)&=&\tau\left(p+U\left(\frac x\alpha\right)+\frac{V\left(\frac x\alpha\right)}p+\frac{W\left(\frac x\alpha\right)}{p^2}+4\eta\log \left|\frac{x^k}{\rho^k}-1\right|\right.\\
&&\left.+O\left(|x|+\alpha+\frac{\beta^k}{\left|x^k-\rho^k\right|}+\rho^k\right)\right)\\
&\le&\tau\left(p+\left(-4+O\left(\frac1p\right)\right)\log \left|\frac x\alpha\right|+4k\eta\log \left(\frac{|x|}\rho+1\right)+O(1)\right)\\
&\le&p\tau\left(1-\frac{4-\sigma}p\log \left|\frac x\alpha\right|+\frac{4k\eta}p\log \left(\frac{|x|}\rho+1\right)\right),
\end{eqnarray*}
for any small $\sigma>0$, and, similarly,
\begin{eqnarray*}
\Upsilon(x)&\ge&-\tau\eta\left(p+U\left(\frac{x^k-\rho^k}{\beta^k}\right)+\frac{V\left(\frac{x^k-\rho^k}{\beta^k}\right)}p+\frac{W\left(\frac{x^k-\rho^k}{\beta^k}\right)}p\right.\\
&&\left.+\frac4\eta\log \left|\frac x\rho\right|+O\left(|x|+\frac\alpha{|x|}+\beta^k+\rho^k\right)\right)\\
&\ge&-p\tau\eta\left(1-\frac{4-\sigma}p\log \left|\frac{x^k-\rho^k}{\beta^k}\right|+\frac4{k\eta p}\log \left(\frac{\left|x^k-\rho^k\right|}{\rho^k}+1\right)\right);
\end{eqnarray*}
therefore, in view of the elementary inequality $\left(1+\frac xp\right)^p\le e^x$, we get
\begin{eqnarray*}
|\Upsilon(x)|^p&\le&(p\tau)^p\max\left\{1-\frac{4-\sigma}p\log \left|\frac x\alpha\right|+\frac{4k\eta}p\log \left(\frac{|x|}\rho+1\right),\right.\\
&&\left.\eta\left(\left(1-\frac{4-\sigma}p\log \left|\frac{x^k-\rho^k}{\beta^k}\right|\right)+\frac4{k\eta p}\log \left(\frac{\left|x^k-\rho^k\right|}{\rho^k}+1\right)\right)\right\}^p\\
&\le&\max\left\{(p\tau)^p\frac{\alpha^{4-\sigma}}{|x|^{4-\sigma}}\left(\frac{|x|}\rho+1\right)^{4k\eta},\right.\\
&&\left.(p\tau\eta)^p\frac{\beta^{(4-\sigma)k}}{\left|x^k-\rho^k\right|^{4-\sigma}}\left(\frac{\left|x^k-\rho^k\right|}{\rho^k}+1\right)^\frac4{k\eta}\right\}\\
&\le&\frac Cp\left(\frac{\alpha^{2-\sigma}}{|x|^{4-\sigma}}\left(\frac{|x|^{4k\eta}}{\rho^{4k\eta}}+1\right)+\rho^{2k-2}\frac{\beta^{(2-\sigma)k}}{\left|x^k-\rho^k\right|^{4-\sigma}}\left(\frac{\left|x^k-\rho^k\right|^\frac4{k\eta}}{\rho^\frac4\eta}+1\right)\right).
\end{eqnarray*}
Now, since \eqref{eta1} yields $4k\eta-2+\sigma+4\delta>0$, then
$$\frac{\alpha^{2-\sigma}}{|x|^{4-\sigma}}\frac{|x|^{4k\eta}}{\rho^{4k\eta}}=|x|^{4k\eta-2+\sigma+4\delta}\frac{\alpha^{2-\sigma-4\delta}}{\rho^{4k\eta}}\frac{\alpha^{4\delta}}{|x|^{2+4\delta}}\le C\frac{\alpha^{2-\sigma-4\delta}}{\rho^{4k\eta}}\frac{\alpha^{4\delta}}{|x|^{2+4\delta}};$$
thanks to the estimate \eqref{alpharhok} and the exponential decay of $\alpha$, we can take $\delta,\sigma$ so that $\frac{\alpha^{2-\sigma-4\delta}}{\rho^{4k\eta}}\le1$. Similarly, from \eqref{eta1} we have $\frac4{k\eta}-2+\sigma+4\delta<0$ for small $\sigma,\delta$, hence
\begin{eqnarray*}
&&\rho^{2k-2}\frac{\beta^{(2-\sigma)k}}{\left|x^k-\rho^k\right|^{4-\sigma}}\frac{\left|x^k-\rho^k\right|^\frac4{k\eta}}{\rho^\frac4\eta}\\
&\le&\left(2|x|\right)^{2k-2}\left|x^k-\rho^k\right|^{\frac4{k\eta}-2+\sigma+4\delta}\frac{\beta^{(2-\sigma-4\delta)k}}{\rho^\frac4\eta}\frac{\beta^{4k\delta}}{\left|x^k-\rho^k\right|^{2+4\delta}}\\
&\le&C\frac{\beta^{\frac{4k}\eta}}{\rho^\frac4\eta}\frac{\beta^{4k\delta}|x|^{2k-2}}{\left|x^k-\rho^k\right|^{2+4\delta}},\\
&\le&C\frac{\beta^{4k\delta}|x|^{2k-2}}{\left|x^k-\rho^k\right|^{2+4\delta}},
\end{eqnarray*}
because we have
$$|x|\ge\rho-\sqrt[k]{\left|x^k-\rho^k\right|}\ge\rho-\beta\ge\frac\rho2.$$
In conclusion, if $2-\sigma\le4\delta$, then
$$|\Upsilon(x)|^p\le\frac Cp\left(\frac{\alpha^{4\delta}}{|x|^{2+4\delta}}+\frac{\beta^{4k\delta}|x|^{2k-2}}{\left|x^k-\rho^k\right|^{2+4\delta}}\right)$$
which, together with \eqref{out1} and \eqref{out2}, concludes the proof.
\end{proof}

\subsection{The size of the error in a suitable norm} 
\label{Section:errorTotale}
For $h\in L^\infty(\Omega)$, we consider the weighted norm
\begin{equation}\label{norm}
\|h\|_{*}:=\sup_{x\in\Omega}\left|\frac{h(x)}{\frac{\alpha^{2\delta}}{\left(|x|^2+\alpha^2\right)^{1+\delta}}+\frac{\beta^{2k\delta}|x|^{2k-2}}{\left(|x^k-\rho^k|^2+\beta^{2k}\right)^{1+\delta}}}\right|,
\end{equation}
where $\delta\in(0,1)$ is given by Lemma \ref{out}.

This type of weighted norm is rather convenient to deal with concentration phenomena in planar domains (see, for instance \cite{EspositoMussoPistoiaJDE2006,pi-ric}).
A fundamental property of this norm is that, since
$$e^{U_1(x)}\le8\frac{\alpha^{2\delta}}{\left(|x|^2+\alpha^2\right)^{1+\delta}},\qquad|x|^{2k-2}e^{U_2(x)}\le8k^2\frac{\beta^{2k\delta}|x|^{2k-2}}{\left(|x^k-\rho^k|^2+\beta^{2k}\right)^{1+\delta}},$$
then
\begin{equation}\label{normbubble}
\left\|e^{U_1(x)}\right\|_{*}+\left\||x|^{2k-2}e^{U_2(x)}\right\|_{*}\le C.
\end{equation}
Another crucial property is that for any $h\in L^\infty(\Omega)$ one has
\begin{equation}\label{normint}
\int_\Omega|h|\le\int_\Omega\|h\|_*\left(\frac{\alpha^{2\delta}}{\left(|x|^2+\alpha^2\right)^{1+\delta}}+\frac{\beta^{2k\delta}|x|^{2k-2}}{\left(|x^k-\rho^k|^2+\beta^{2k}\right)^{1+\delta}}\right)\mathrm dx\le C\|h\|_*
\end{equation}

\begin{proposition}\label{r} 
Let $\Upsilon$ be the function in \eqref{upsp}, and assume that the parameter
$\eta$ satisfies \eqref{eta} and that $\alpha, \beta, \rho,\tau$ are as in \eqref{param}. Let $\mathscr E
=\Delta\Upsilon+|\Upsilon|^{p-1}\Upsilon, $ be the error term in the reduction procedure (see Subsection \ref{SubsectionReduction}). Then there exist $C,\delta>0$ such that 
$$\|\mathscr E\|_*\le\frac C{p^4}.$$
\end{proposition}\
\begin{proof} 
We fix $\theta\in(0,\theta_0]$, with $\theta_0$ is as in Lemma \ref{alphabetarho}, small enough so that both Lemmas \ref{in1} and \ref{in2} hold true; we also fix $\delta$ as in Lemma \ref{out}.\\
If $\left|\frac x\alpha\right|\le\frac1{\alpha^\theta}$, then, by Lemma \ref{in1},
\begin{equation}\label{rin1}
|\mathscr E(x)|\le\frac C{\alpha^2p^4}\frac{\log ^6\left(\frac{|x|}\alpha+2\right)}{\left(\frac{|x|^2}{\alpha^2}+1\right)^2}\le\frac C{\alpha^2p^4}\frac{\left(\frac{|x|^2}{\alpha^2}+1\right)^{1-\delta}}{\left(\frac{|x|^2}{\alpha^2}+1\right)^2}=\frac C{p^4}\frac{\alpha^{2\delta}}{\left(|x|^2+\alpha^2\right)^{1+\delta}}.
\end{equation}
Similarly, if $\left|\frac{x^k-\rho^k}{\beta^k}\right|\le\frac1{\beta^{k\theta}}$, then Lemma \ref{in2} yields
\begin{equation}\label{rin2}
|\mathscr E(x)|\le C\frac{\rho^{2k-2}}{\beta^{2k}p^4}\frac{\log ^6\left(\frac{\left|x^k-\rho^k\right|}{\beta^k}+2\right)}{\left(\frac{\left|x^k-\rho^k\right|^2}{\beta^{2k}}+1\right)^2}\le\frac C{p^4}\frac{\beta^{2k\delta}|x|^{2k-2}}{\left(\left|x^k-\rho^k\right|^2+\beta^{2k}\right)^{1+\delta}},
\end{equation}
since, by \eqref{betarho}, $\beta^{1-\theta}=o(\rho)$, hence $|x|\ge\rho+o(\rho)$.
Finally, if $\left|\frac x\alpha\right|>\frac1{\alpha^\theta}$ and $\left|\frac{x^k-\rho^k}{\beta^k}\right|>\frac1{\beta^{k\theta}}$, then $|x|>\alpha$ and $\left|x^k-\rho^k\right|>\beta^k$, since $\alpha,\beta\to0$, therefore Lemma \ref{out} implies:
\begin{eqnarray}
\nonumber|\mathscr E(x)|&\le&\frac Cp\left(\frac{\alpha^{4\delta}}{|x|^{2+4\delta}}+\frac{\beta^{4k\delta}|x|^{2k-2}}{\left|x^k-\rho^k\right|^{2+4\delta}}\right)\\
\nonumber&\le&\frac Cp\left(\frac{\alpha^{2(1+\theta)\delta}}{|x|^{2+2\delta}}+\frac{\beta^{2k(1+\theta)\delta}|x|^{2k-2}}{\left|x^k-\rho^k\right|^{2+2\delta}}\right)\\
\label{rout}&\le&Ce^{-\frac{\theta\delta}2p}\left(\frac{\alpha^{2\delta}}{\left(|x|^2+\alpha^2\right)^{1+\delta}}+\frac{\beta^{2k\delta}|x|^{2k-2}}{\left(|x^k-\rho^k|^2+\beta^{2k}\right)^{1+\delta}}\right),
\end{eqnarray}
since, by Lemma \ref{asymptotic}, $\alpha,\beta^k\le e^{-\frac p4}$. From \eqref{rin1}, \eqref{rin2}, \eqref{rout} and the definition \eqref{norm} of $\|\cdot\|_{*}$ we get the desired estimate for $\mathscr E$.
\end{proof}

\section{The linear theory}\label{4}
 This section is devoted to prove that the linear operator $\mathscr L$ introduced in \eqref{elle} is invertible in a subspace of $H^1_{0,k}(\Omega)$ of codimension $1$.\\
 
 The proof is quite complex, hence the section is structured in multiple subsections.\\

In order to state the invertibility result, let us
first remind some well known facts about the solutions of the linearized equations in the whole plane (see for example \cite{bp,grossi-pistoia-arma}) and fix some notations.\\

Let $U_{1}$ be the profile introduced in Subsection \ref{subsection:ansatzFirst}. It is known that all the solutions of the linear equation
$$
-\Delta Z (x)=e^{U_1(x)}Z (x),\qquad x\in\mathbb R^2
$$
are linear combination of
\begin{eqnarray}
Z_{0,1}(x)&:=&Z_0\left(\frac x\alpha\right)=\frac{|x|^2-\alpha^2}{|x|^2+\alpha^2}, \label{Z_01}\\
Z_{1,1}(x)&:=&Z_1\left(\frac x\alpha\right)=\frac{4\alpha x_1}{|x|^2+\alpha^2}, \\
Z_{2,1}(x)&:=&Z_2\left(\frac x\alpha\right)=\frac{4\alpha x_2}{|x|^2+\alpha^2},
\end{eqnarray}
where
\begin{equation}\label{def:Z_i}Z_0(x):=\frac{|x|^2-1}{|x|^2+1},\quad Z_1(x):=\frac{4x_1}{|x|^2+1},\quad Z_2(x):=\frac{4x_2}{|x|^2+1}.\end{equation}
We observe that
\begin{equation}
\label{ortz}
\int _{\mathbb R^2}e^U Z_iZ_j=0\ \hbox{ if }\ i,j=0,1,2,\ i\not=j,
\end{equation}
where $U$ is the bubble defined in \eqref{bubble}.\\
Furthermore, let $U_{2}$ be the profile introduced in Subsection \ref{subsection:ansatzFirst}. It is also known that all the solutions of the linear equation
$$
-\Delta Z (x)=|x|^{2k-2}e^{U_2(x)}Z (x),\qquad x=(x_1,x_2)\in\mathbb R^2
$$
are linear combination of
\begin{eqnarray}
Z_{0,2}(x)&:=&Z_0\left(\frac{x^k-\rho^k}{\beta^k}\right)=\frac{\left|x^k-\rho^k\right|^2-\beta^{2k}}{\left|x^k-\rho^k\right|^2+\beta^{2k}},\label{Z_02}\\
 Z_{1,2}(x)&:=&Z_1\left(\frac{x^k-\rho^k}{\beta^k}\right)=\frac{4\beta^k\left(\Re\left(x^k\right)-\rho^k\right)}{\left|x^k-\rho^k\right|^2+\beta^{2k}},\\
Z_{2,2}(x)&:=&Z_2\left(\frac{x^k-\rho^k}{\beta^k}\right)=\frac{4\beta^k\Im\left(x^k\right)}{\left|x^k-\rho^k\right|^2+\beta^{2k}}.
\end{eqnarray}
 We set \begin{equation}\label{def:Ztilde}\widetilde Z(x):= Z_{1,2}(x),\end{equation} and we introduce the orthogonality condition which has to be satisfied by the remainder term $\phi$ in the ansatz \eqref{uphi}:
\begin{equation}\label{orto}
\int_{\Omega}\nabla P\widetilde Z\cdot\nabla\phi = \int_\Omega|x|^{2k-2}e^{U_2(x)}\widetilde Z(x)\phi(x)\mathrm dx=0
\end{equation}

We also set the notation $L^p_k(\Omega)$, for $p\in[1,\infty]$, for the space of the $k$-symmetric functions in $L^p(\Omega)$, namely:
\begin{equation}\label{linftyk}
L^p_k(\Omega):=\left\{u\in L^p(\Omega):\,u(x)=u\left(\overline x\right)=u\left(e^{\mathrm i\frac{2\pi}k}x\right)\,\mbox{for a.e. }x\in\Omega\right\}.
\end{equation}
and recall that the space $H^1_{0,k}(\Omega)$, of the $k$-symmetric functions in $H^1_0(\Omega)$, has been similarly defined in \eqref{def:H10sim}.
 Let us observe that $H^1_{0,k}(\Omega)$ splits as the direct sum of its subspaces
$$\mathbf Z:= \left\{c\mathrm P\widetilde Z:\,c\in\mathbb R\right\}\ \hbox{and}\
\mathbf Z^\perp:=\left\{\phi\in H^1_{0,k}(\Omega):\,\int_\Omega|x|^{2k-2}e^{U_2(x)}\widetilde Z(x)\phi(x)\mathrm dx=0\right\},
$$
Now, we are ready to state the main result of this section, concerning the invertibility of $\mathscr L$ in the space $\mathbf Z^\perp$.

\begin{proposition}\label{linear} 
Let $\mathscr L$ be the linear operator introduced in \eqref{elle} and let us assume that the parameters $\alpha, \beta, \rho,\tau, \eta$ in the definition of $\Upsilon$ are choosen as in \eqref{param}, and that $\eta$ satisfies \eqref{eta}.
Then, there exists $C>0$ such that, for any $ h\in L^\infty_k(\Omega)$, the linear problem 
\begin{equation}\label{eqphih}
\left\{\begin{array}{ll}\mathscr L\phi(x)=h(x)+\mathfrak c|x|^{2k-2}e^{U_2(x)}\widetilde Z(x)&x\in\Omega\\\phi(x)=0&x\in\partial\Omega.\end{array}\right.
\end{equation}
 has a unique solution $(\phi, \mathfrak c)\in H^1_{0,k}(\Omega)\cap L^\infty_k(\Omega)\times \mathbb R$ such that $\phi$ satisfies \eqref{orto} and 
\begin{equation}\label{expressionC}\mathfrak c=-\frac{\int_{\Omega}\nabla h \cdot\nabla P\widetilde Z}{\int_{\Omega} \left|\nabla P\widetilde Z\right|^2}.\end{equation}
Furthermore
\begin{equation}\label{estimphi}\|\phi\|_\infty+|\mathfrak c|\le Cp\|h\|_*,
\end{equation}
where the weighted norm $\|\cdot\|_*$ is defined in \eqref{norm}. 
\end{proposition}

The proof of Proposition \ref{linear} is postponed at the end of the section. We will need some lemmas, concerning the properties of the linearized operator $\mathscr L$. We stress that, since by assumption the parameters $\eta$ and $\alpha, \beta, \rho,\tau$ in satisfy \eqref{eta} and \eqref{param}, we can exploit all the estimates collected in Section \ref{3}.

\subsection{The estimate of the non-autonomous term in $\mathscr L$}
\begin{lemma}\label{derivestim} 
There exist $\theta,C>0$ such that if $|y|\le\frac1{\alpha^\theta}$ then
\begin{eqnarray}
\label{deriv1}p|\Upsilon(\alpha y)|^{p-1}&=&\frac{e^{U(y)}}{\alpha^2}\left(1+\frac1p\left(V(y)-U(y)-\frac{U(y)^2}2\right)\right.\\
\nonumber&&\left.+O\left(\frac{\log ^4(|y|+2)}{p^2}\right)\right),
\end{eqnarray}
and if $|y|\le\frac1{\beta^{k\theta}}$ then
\begin{eqnarray}
\label{deriv2} p\left|\Upsilon\left(\sqrt[k]{\beta^ky+\rho^k}\right)\right|^{p-1}&=&\frac{\rho^k}{\beta^{2k}}e^{U(y)}\left(1+\frac1p\left(V(y)-U(y)-\frac{U(y)^2}2\right)\right.\\
\nonumber&&\left.+O\left(\frac{\log ^4(|y|+2)}{p^2}\right)\right).
\end{eqnarray}
Moreover, there exists $C>0$ such that
\begin{eqnarray}
\label{deriv}p|\Upsilon(x)|^{p-1}&\le&C\left(e^{U_1(x)}+|x|^{2k-2}e^{U_2(x)}\right).
\end{eqnarray}
\end{lemma}
\begin{proof}[Proof of Lemma \ref{derivestim}] 
If $|y|\le\frac1{\alpha^\theta}$, then from \eqref{estimuvw} and \eqref{taylorp} we get
\begin{eqnarray*}
p|\Upsilon(\alpha y)|^{p-1}&=&p\left(\tau\left(p+U(y)+\frac{V(y)}p+O\left(\frac{\log (|y|+2)}{p^2}\right)\right)\right)^{p-1}\\
&=&\frac1{\alpha^2}\left(1+\frac{U(y)}p+\frac{V(y)}{p^2}+O\left(\frac{\log (|y|+2)}{p^3}\right)\right)^{p-1}\\
&=&\frac{e^{U(y)}}{\alpha^2}\left(1+\frac1p\left(V(y)-\frac{U(y)^2}2\right)+O\left(\frac{\log ^4(|y|+2)}{p^2}\right)\right)\\
&&\cdot\left(1-\frac{U(y)}p+O\left(\frac{\log (|y|+2)}{p^2}\right)\right)\\
&=&\frac{e^{U(y)}}{\alpha^2}\left(1+\frac1p\left(V(y)-U(y)-\frac{U(y)^2}2\right)+O\left(\frac{\log ^4(|y|+2)}{p^2}\right)\right);
\end{eqnarray*}
Estimate \eqref{deriv2} can be shown similarly.\\
In order to prove \eqref{deriv}, we observe that, due to \eqref{pu1}, \eqref{pu2} and Lemma \ref{asymptotic}, one has
$$\Upsilon(x)=\tau\left(1+O_{\le0}\left(\frac1p\right)\right)\log \frac{\left(\left|x^k-\rho^k\right|^2+\beta^{2k}\right)^{2\eta}}{\left(|x|^2+\alpha^2\right)^2}+O(1).$$
If $|x|\le\frac\rho2$, then $\frac{\rho^{2k}}C\le\left|x^k-\rho^k\right|^2+\beta^{2k}\le C\rho^{2k}$ for some $C>0$; therefore, due to \eqref{param} one gets
\begin{eqnarray*}
\Upsilon(x)&=&\tau\left(1+O_{\le0}\left(\frac1p\right)\right)\log \frac{\rho^{4k\eta}}{\left(|x|^2+\alpha^2\right)^2}+O\left(\frac1p\right)\\
&=&p\tau+\tau\left(1+O_{\le0}\left(\frac1p\right)\right)\log \frac{\alpha^4}{\left(|x|^2+\alpha^2\right)^2}+O\left(\frac1p\right),
\end{eqnarray*}
hence
\begin{eqnarray*}
p|\Upsilon(x)|^{p-1}&=&O\left(p^p\tau^{p-1}\left(1+\left(\frac1p+O_{\le0}\left(\frac1{p^2}\right)\right)\log \frac{\alpha^4}{\left(|x|^2+\alpha^2\right)^2}\right)^{p-1}\right)\\
&=&O\left(\frac1{\alpha^2}e^{\left(1+O_{\le0}\left(\frac1p\right)\right)\log \frac{\alpha^4}{\left(|x|^2+\alpha^2\right)^2}\left(1-\frac1p\right)}\right)\\
&=&O\left(e^{U_1(x)}\right).
\end{eqnarray*}
On the other hand, if $\frac\rho2<|x|\le2\rho$, then $\frac{\rho^2}C\le|x|^2+\alpha^2\le C\rho^2$, therefore by interchanging the two bubbles one shows that $p|\Upsilon(x)|^{p-1}=O\left(|x|^{2k-2}e^{U_2(x)}\right)$.\\
Finally, if $|x|>2\rho$, then
$$\frac{|x|^2}C\le|x|^2+\alpha^2\le C|x|^2,\qquad\qquad\frac{|x|^{2k}}C\le\left|x^k-\rho^k\right|^2+\beta^{2k}\le C|x|^{2k},$$
therefore
$$\Upsilon(x)=\tau\left(1+O\left(\frac1p\right)\right)4(k\eta-1)\log |x|+O\left(\frac1p\right);$$
since we want to estimate $p|\Upsilon(x)|^{p-1}$ by means of $|x|^{2k-2}e^{U_2(x)}$, which decays power-like, we need some basic estimates for products of logarithms and powers.\\
The function $t\mapsto t^{p-1}\log ^{p-1}\frac1t$ is decreasing if $e^{-\frac{p+1}{2k+2}}\le t\le1$; since $\eta$ satisfies \eqref{etak}, then from Lemma \ref{asymptotic} we obtain $e^{-\frac{p+1}{2k+2}}\le C\rho$ for some $C>0$, therefore for $|x|>2\rho$ we get the estimate
$$\log ^{p-1}\frac2{C|x|}\le(2\rho)^{2k+2}\left(\log \frac1{C\rho}\right)^{p-1}\frac1{|x|^{2k+2}}.$$
Therefore, using the elementary inequality
$$te^p\left(\frac{\log \frac1t}p\right)^p\le1\qquad\forall\,x\in(0,1],$$
we get:
\begin{eqnarray*}
p|\Upsilon(x)|^{p-1}&=&p^{p-1}\tau^p\left(1+O\left(\frac1p\right)\right)^{p-1}\left(\frac{4(k\eta-1)\log \frac1{|x|}+O(1)}p\right)^{p-1}\\
&=&O\left(\frac{\rho^{2k-2}}{\beta^{2k}}\left(\frac{4\left(k-\frac1\eta\right)\log \frac1{|x|}+O(1)}p\right)^{p-1}\right)\\
&=&O\left(\frac{\rho^{4k}}{\beta^{2k}}\left(\frac{4\left(k-\frac1\eta\right)\log \frac1\rho+O(1)}p\right)^{p-1}\frac1{|x|^{2k+2}}\right)\\
&=&O\left(\rho^{4\left(k-\frac1\eta\right)}e^p\left(\frac{4\left(k-\frac1\eta\right)\log \frac1\rho+O(1)}p\right)^p\frac{\beta^{2k}}{|x|^{2k+2}}\right)\\
&=&O\left(\frac{\beta^{2k}}{|x|^{2k+2}}\right)\\
&=&O\left(|x|^{2k-2}e^{U_2(x)}\right).
\end{eqnarray*}
\end{proof}
\begin{remark}\label{normphi} 
Multiplying equation \eqref{eqphih} by $\phi$ and integrating by parts, we get
$$\|\phi\|^2_{H^1_0}=\int_\Omega p|\Upsilon|^{p-1}\phi^2-\int_\Omega h\phi,$$
therefore in view of \eqref{deriv} one has
\begin{eqnarray*}
\|\phi\|^2_{H^1_0}&\le&C\|\phi\|_\infty^2\int_\Omega\left(e^{U_1(x)}+|x|^{2k-2}e^{U_2(x)}\right)\mathrm dx\\
&+&C\|\phi\|_\infty\|h\|_*\int_\Omega\left(\frac{\alpha^{2\delta}}{\left(|x|^2+\alpha^2\right)^{1+\delta}}+\frac{\beta^{2k\delta}|x|^{2k-2}}{\left(|x^k-\rho^k|^2+\beta^{2k}\right)^{1+\delta}}\right)\mathrm dx\\
&\le&C\left(\|\phi\|^2_\infty+\|\phi\|_\infty\|h\|_*\right),
\end{eqnarray*} that is
$$\|\phi\|_{H^1_0}\le C\left(\|\phi\|_\infty+\|h\|_*\right).$$ 
\end{remark}

\subsection{ $\mathscr L$ satisfies the maximum principle far from the peaks}
\begin{lemma}\label{maxprinc} 
For some $R>0$, independent on $p$, the operator $\mathscr L$ satisfies the maximum principle on the domain
$$\widetilde\Omega:=\left\{x\in\Omega:\,\left|\frac x\alpha\right|>R,\left|\frac{x^k-\rho^k}{\beta^k}\right|>R\right\};$$
in other words,
$$\left\{\begin{array}{ll}\mathscr L\psi\le0&\mbox{in }\widetilde\Omega\\\psi\ge0&\mbox{on }\partial\widetilde\Omega\end{array}\right.\qquad\Rightarrow\qquad\psi\ge0\mbox{ in }\widetilde\Omega.$$
\end{lemma}
\begin{proof} 
We suffice to construct a strictly positive \emph{super-solution} $\zeta>0$ in $\widetilde\Omega$ such that $\mathscr L\zeta<0$ in $\widetilde\Omega$. In fact, if there existed some $\psi$ satisfying
$$\left\{\begin{array}{ll}\mathscr L\psi\le0&\mbox{in }\widetilde\Omega\\\psi<0&\mbox{in }\omega\\\psi=0&\mbox{on }\partial\omega\end{array}\right.,$$
for some region $\omega\subset\widetilde\Omega$, then integrating by parts on $\omega$ one would get the following contradiction:
$$0<\int_\omega\psi\mathscr L\zeta=\int_\omega\zeta\mathscr L\psi-\int_{\partial\omega}\zeta\partial_\nu\psi\le0.$$
We define
$$\zeta(x):=1-2C\frac{\alpha^2}{|x|^2}-2C\frac{\beta^{2k}}{\left|x^k-\rho^k\right|^2},$$
with $C>0$ given by \eqref{deriv}. By construction, $\zeta>0$ in $\widetilde\Omega$ if $R\ge\sqrt2C$; moreover,
\begin{eqnarray*}
\Delta\zeta(x)&=&-8C\left(\frac{\alpha^2}{|x|^4}+\frac{k^2\beta^{2k}|x|^{2k-2}}{\left|x^k-\rho^k\right|^4}\right)\\
&<&-C\left(e^{U_1(x)}+|x|^{2k-2}e^{U_2(x)}\right)\\
&<&-C\left(e^{U_1(x)}+|x|^{2k-2}e^{U_2(x)}\right)\zeta(x),
\end{eqnarray*}
since $\zeta<1$. Therefore, from \eqref{deriv} we get:
$$\mathscr L\zeta(x)=\Delta\zeta(x)+p|\Upsilon(x)|^{p-1}\zeta(x)<0.$$
\end{proof}

\subsection{A partial uniform estimate of the solutions of \eqref{eqphih}, when $\mathfrak c=0$}
Using Lemma \ref{maxprinc} and Lemma \ref{derivestim}, we deduce the following partial estimate, for solutions $\phi$ of the linear problem \eqref{eqphih}, when $\mathfrak c=0$.
\begin{lemma}\label{intnorm} 
There exists some $C>0$ such that for any $h\in L^\infty_k(\Omega)$, if $\phi\in H^1_{0,k}(\Omega)\cap L^\infty_k(\Omega)$ solves
\begin{equation}\label{lphih}
\left\{\begin{array}{ll}\mathscr L\phi=h&\mbox{in }\Omega\\\phi=0&\mbox{on }\partial\Omega\end{array}\right.,
\end{equation}
then
$$\|\phi\|_\infty\le C(\|\phi\|_{L^{\infty}(\Omega\setminus\widetilde\Omega)}+\|h\|_*),$$
where $\|\cdot\|_*$ is the weighted norm defined in \eqref{norm}, 
 and $\widetilde\Omega$ is as in Lemma \ref{maxprinc}. Moreover $H^1_{0,k}(\Omega)$ and $L^\infty_k(\Omega)$
 are the spaces of symmetric functions defined in \eqref{def:H10sim} and \eqref{linftyk}, respectively. 
\end{lemma}

\begin{proof} 
Let $\psi_i$ be the solutions to the following problems, for $j=1,2$:
\begin{equation}\label{eqpsi1}
\left\{\begin{array}{ll}-\Delta\psi_1(x)=\frac{\alpha^{2\delta}}{|x|^{2+2\delta}}&R\alpha<|x|<M_1:=2\,\mathrm{diam}(\Omega)\\\psi_1(x)=0&|x|=R\alpha,|x|=M_1\end{array}\right.,
\end{equation}
\begin{equation}\label{eqpsi2}
\left\{\begin{array}{ll}-\Delta\psi_2(x)=\frac{\beta^{2k\delta}|x|^{2k-2}}{\left|x^k-\rho^k\right|^{2+2\delta}}&R\beta^k<\left|x^k-\rho^k\right|<M_2:=2\,\mathrm{diam}\left(\Omega^k\right)\\\psi_2(x)=0&\left|x^k-\rho^k\right|=R\beta^k,\left|x^k-\rho^k\right|=M_2\end{array}\right.;
\end{equation}
where $\delta\in (0,1)$ is the one fixed in the definition of the weighted norm $\|\cdot\|_*$ (see \eqref{norm}).
Their explicit form is given by
\begin{eqnarray*}
\psi_1(x)&=&-\frac{\alpha^{2\delta}}{4\delta^2}\frac1{|x|^{2\delta}}+A_1+B_1\log |x|,\\
\psi_2(x)&=&-\frac{\beta^{2k\delta}}{4k^2\delta^2}\frac1{\left|x^k-\rho^k\right|^{2\delta}}+A_2+B_2\log \left|x^k-\rho^k\right|,
\end{eqnarray*}
with
\begin{eqnarray*}
A_1&=&\frac{\alpha^{2\delta}}{4\delta^2M_1^{2\delta}}-B_1\log M_1;\\
B_1&=&\frac1{4\delta^2}\left(\frac{\alpha^{2\delta}}{M_1^{2\delta}}-\frac1{R^{2\delta}}\right)\frac1{\log \frac{M_1}{R\alpha}}<0,\\
A_2&=&\frac{\beta^{2k\delta}}{4k^2\delta^2M_2^{2\delta}}-B_1\log \left(M_2\right),\\
B_2&=&\frac1{4k^2\delta^2}\left(\frac{\beta^{2k\delta}}{M_2^{2\delta}}-\frac1{R^{2\delta}}\right)\frac1{\log \frac{M_2}{R\beta^k}}<0.
\end{eqnarray*}
Both functions are positive, due to the maximum principle for $-\Delta$, and uniformly bounded from above, independently on $p$, since
\begin{eqnarray*}
\psi_1(x)&\le&A+B\log (R\alpha)=\frac{\alpha^{2\delta}}{4\delta^2M_1^{2\delta}}-B_1\log \frac{M_1}{R\alpha}=\frac1{4\delta^2R^{2\delta}},\\
\psi_2(x)&\le&A+B\log \left(R\beta^k\right)=\frac1{4k^2\delta^2R^{2\delta}};
\end{eqnarray*}
moreover, due to the choice of $M_1,M_2$, we have
$$|x|<M_1,\qquad\left|x^k-\rho^k\right|<M_2,\qquad\forall\,x\in\widetilde\Omega,$$
so $\psi_1,\psi_2$ respectively solve \eqref{eqpsi1}, \eqref{eqpsi2} on $\widetilde\Omega$.\\
Now, we fix $h,\phi$ solving \eqref{lphih} and define
$$\widetilde\phi:=2\|\phi\|_{L^{\infty}(\Omega\setminus\widetilde\Omega)}\zeta+2\|h\|_*(\psi_1+\psi_2),$$
$\zeta$ being the barrier function introduced in the proof of Lemma \ref{maxprinc}. After possibly choosing a larger $R\ge2\sqrt{2C}$, with $C$ as in \eqref{deriv}, one has $\zeta\ge\frac12$ on $\widetilde\Omega$, therefore if either $|x|=R\alpha$ or $\left|x^k-\rho^k\right|=R\beta^k$ then
$$\widetilde\phi(x)\ge2\|\phi\|_{L^{\infty}(\Omega\setminus\widetilde\Omega)}\zeta(x)\ge\|\phi\|_{L^{\infty}(\Omega\setminus\widetilde\Omega)}\ge|\phi(x)|,$$
whereas on $\partial\Omega$ one has $\widetilde\phi\ge0=\phi$, since $\zeta,\psi_1,\psi_2$ are all positive. On the other hand, since $\mathscr L\zeta<0$, then for any $x\in\widetilde\Omega$
\begin{eqnarray*}
\mathscr L\widetilde\phi(x)&\le&2\|h\|_*(\mathscr L\psi_1+\mathscr L\psi_2)\\
&=&2\|h\|_*\left(-\frac{\alpha^{2\delta}}{|x|^{2+2\delta}}-\frac{\beta^{2k\delta}|x|^{2k-2}}{\left|x^k-\rho^k\right|^{2+2\delta}}+p|\Upsilon|^{p-1}(x)(\psi_1(x)+\psi_2(x))\right)\\
&\le&2\|h\|_*\left(-\frac{\alpha^{2\delta}}{\left(|x|^2+\alpha^2\right)^{1+\delta}}-\frac{\beta^{2k\delta}|x|^{2k-2}}{\left(\left|x^k-\rho^k\right|^2+\beta^{2k}\right)^{1+\delta}}\right.\\
&&\left.+(\|\psi_1\|_\infty+\|\psi_2\|_\infty)\left(e^{U_1(x)}+|x|^{2k-2}e^{U_2(x)}\right)\right)\\
&\le&-2\|h\|_*\left(1-\frac{4k^2C}{\delta^2R^{2\delta}}\right)\left(\frac{\alpha^{2\delta}}{\left(|x|^2+\alpha^2\right)^{1+\delta}}+\frac{\beta^{2k\delta}|x|^{2k-2}}{\left(\left|x^k-\rho^k\right|^2+\beta^{2k}\right)^{1+\delta}}\right)\\
&\le&-|h(x)|\\
&=&-|\mathscr L\phi(x)|,
\end{eqnarray*}
provided $R$ is large enough.\\
The maximum principle from Lemma \ref{maxprinc} implies $|\phi(x)|\le\widetilde\phi(x)$ for all $x\in\widetilde\Omega$, therefore
$$\|\phi\|_\infty\le\left\|\widetilde\phi\right\|_\infty\le2\|\phi\|_{L^{\infty}(\Omega\setminus\widetilde\Omega)}\|\zeta\|_\infty+2\|h\|_*(\|\psi_1\|_\infty+\|\psi_2\|_\infty)\le2\|\phi\|_{L^{\infty}(\Omega\setminus\widetilde\Omega)}+\frac1{\delta^2R^{2\delta}}\|h\|_*.$$
\end{proof}

\subsection{A uniform estimate of the solutions of \eqref{eqphih}, when $\mathfrak c=0$}
Using Lemma \ref{intnorm} and Lemma \ref{derivestim}, we deduce the following uniform estimate, for solutions $\phi$ of the linear problem  \eqref{eqphih}, when $\mathfrak c=0$, and under the orthogonality condition \eqref{orto}.
\begin{lemma}\label{orthog2} 
There exists some $C>0$ such that, for any $h\in L^\infty_k(\Omega)$, if $\phi\in H^1_{0,k}(\Omega)\cap L^\infty_k(\Omega)$ solves
\begin{equation}\label{eqorthog2}
\left\{\begin{array}{ll}\mathscr L\phi=h&\mbox{in }\Omega\\\phi=0&\mbox{on }\partial\Omega\\\int_\Omega|x|^{2k-2}e^{U_2(x)}\widetilde Z(x)\phi(x)\mathrm dx=0\end{array}\right.,
\end{equation}
then
$$\|\phi\|_\infty\le Cp\|h\|_*.$$
Here $\widetilde Z$ is the function in \eqref{def:Ztilde}, $\|\cdot\|_*$ is the weighted norm defined in \eqref{norm}, and $H^1_{0,k}(\Omega)$ and $L^\infty_k(\Omega)$
 are the spaces of symmetric functions defined in \eqref{def:H10sim} and \eqref{linftyk}, respectively.\end{lemma}

\begin{proof} 
Suppose, by contradiction, there exists sequences $\phi_n,h_n$ solving \eqref{eqorthog2} for $p_n\to+\infty$ with $\|\phi_n\|_\infty=1$ and $ {p_n}\|h_n\|_*\to0$.\\
Then, the rescalements
$$
\phi_{n,1}(y):=\phi_n(\alpha y),\qquad\phi_{n,2}(y):=\phi_n\left(\sqrt[k]{\beta^ky+\rho^k}\right)
$$
 solve
$$\Delta\phi_{n,1}(y)+\alpha^2p|\psi (\alpha y)|^{p-1}\phi_{n,1}(y)=\alpha^2 h(\alpha y)\qquad\qquad\mbox{in }\frac{\Omega}{\alpha}$$
and
$$\Delta\phi_{n,2}+\beta^2p\left|\psi \left(\sqrt[k]{\beta^ky+\rho^k}\right)\right|^{p-1}\phi_{n,2}(y)=\beta^2 h\left(\sqrt[k]{\beta^ky+\rho^k}\right)\qquad\qquad\mbox{in }\frac{\Omega^k-\rho^k}{\beta^k}$$
respectively. Hence they 
converge, due to \eqref{deriv1}, \eqref{deriv2} in Lemma \ref{derivestim},
and elliptic estimates, locally uniformly to bounded solutions $\phi_{\infty,1},\phi_{\infty,2}$ to
$$\Delta\phi_{\infty,j}+e^U\phi_{\infty,j} {=0}\qquad\qquad\mbox{in }\mathbb R^2;$$
such solutions are linear combination of the functions $Z_0,Z_1,Z_2$ defined in \eqref{def:Z_i}. We are allowed to pass to the limit in the orthogonality condition in \eqref{eqorthog2}, thanks to the dominated convergence theorem, and we get
$$\int_{\mathbb R^2}e^{U(y)}Z_1(y)\phi_{\infty,2}(y)\mathrm dy=0;$$
moreover, since $\phi\in H^1_{0,k}(\Omega)$, then $\phi_{n,1}\left(e^{\mathrm i\frac{2\pi}k}x\right)=\phi_{n,1}(x)$, then
$$\int_{\mathbb R^2}e^{U(y)}Z_1(y)\phi_{\infty,1}(y)\mathrm dy=\int_{\mathbb R^2}e^{U(y)}Z_1(y)\phi_{\infty,2}(y)\mathrm dy=0,$$
and similarly we get $\phi_{n,2}\left(\overline x\right)=\phi_{n,2}(x)$, hence
$$\int_{\mathbb R^2}e^{U(y)}Z_2(y)\phi_{\infty,2}(y)\mathrm dy=0.$$
Thus, in view of the orthogonality relations \eqref{ortz} we get
$$\phi_{n,1}\to \mathcal C_1Z_0,\qquad\phi_{n,2}\to \mathcal C_2Z_0,\qquad\mbox{in }L^\infty_{\mathrm{loc}}\left(\mathbb R^2\right),$$
for some $\mathcal C_1,\mathcal C_2\in\mathbb R$.
To get a contradiction, we suffice to show that $\mathcal C_1=\mathcal C_2=0$: in fact, this would mean that both $\phi_{n,1}$ and $\phi_{n,2}$ converge to $0$ in $L^\infty_{\mathrm{loc}}\left(\mathbb R^2\right)$, hence Lemma \ref{intnorm} gives:
\begin{eqnarray*}
1&=&\|\phi_n\|_\infty\\
&\le&C(\|\phi_n\|_{L^{\infty}(\Omega\setminus\widetilde\Omega)}+\|h_n\|_*)\\
&=&C\left(\max\left\{\sup_{|x|\le R}|\phi_{n,1}(x)|,\sup_{|x|\le R}|\phi_{n,2}(x)|\right\}+o(1)\right)\\
&\to&0.
\end{eqnarray*}
Let us now define
$$S(x):=\frac23\frac{|x|^2-1}{|x|^2+1}\log \left(|x|^2+1\right)-\frac43\frac1{|x|^2+1},\qquad\qquad\qquad T(x):=\frac2{|x|^2+1},$$
and
\begin{eqnarray*}
Q_1(x)&:=&S\left(\frac x\alpha\right)+\left(\frac43\log \alpha\right)Z_{0,1}(x)+\frac{8\pi}3H(0,0)T\left(\frac x\alpha\right),\\
Q_2(x)&:=&S\left(\frac{x^k-\rho^k}{\beta^k}\right)+\left(\frac43k\log \beta\right)Z_{0,2}(x)+\frac{8\pi}3H(0,0)T\left(\frac{x^k-\rho^k}{\beta^k}\right),
\end{eqnarray*}
where $Z_{0,1}$ and $Z_{0,2}$ are the functions in \eqref{Z_01} and \eqref{Z_02}, respectively.
In view of the asymptotic behaviors
$$S(x)=\frac43\log |x|+O\left(\frac1{|x|}\right),\qquad T(x)=O\left(\frac1{|x|^2}\right)$$
we obtain, uniformly for $x\in\partial\Omega$,
\begin{eqnarray*}
\mathrm PQ_1(x)-Q_1(x)&=&-\frac43\log |x|+O(\alpha),\\
\mathrm PQ_2(x)-Q_2(x)&=&-\frac43k\log |x|+O\left(\beta^k+\rho^k\right);
\end{eqnarray*}
therefore, arguing as in the proof of Lemma \ref{pupvpw}, we get, for $x\in\Omega$,
\begin{eqnarray}
\label{pq1}\mathrm PQ_1(x)&=&Q_1(x)-\frac83\pi H(0,0)+O(|x|+\alpha),\\
\label{pq2}\mathrm PQ_2(x)&=&Q_2(x)-\frac83k\pi H(0,0)+O\left(\left|x^k-\rho^k\right|+\beta^{2k}+\rho^k\right)
\end{eqnarray}
and similarly, for $|y|\le\frac12\frac\rho\alpha$,
$$\mathrm PQ_2(\alpha y)=\frac43k\log \rho-\frac83k\pi H(0,0)+O\left(\frac{\alpha^k}{\rho^k}|y|^k+\frac{\beta^k}{\rho^k}+\rho^k\right),$$
and, for $|y|\le\frac12\frac{\rho^k}{\beta^k}$,
$$\mathrm PQ_1\left(\sqrt[k]{\beta^ky+\rho^k}\right)=\frac43k\log \rho-\frac83\pi H(0,0)+O\left(\frac{\beta^k}{\rho^k}|y|+\frac\alpha\rho+\rho\right).$$
Since $S,T$ are solutions to
$$\Delta S+e^US=e^UZ_0\qquad\mbox{in }\mathbb R^2,\qquad\qquad\qquad\Delta T+e^UT=e^U\qquad\mbox{in }\mathbb R^2,$$
then the $\mathrm PQ_j$'s solve
\begin{equation}\label{eqq1}
\Delta\mathrm PQ_1+p|\Upsilon|^{p-1}Q_1=e^{U_1}Z_{0,1}+\left(p|\Upsilon|^{p-1}-e^{U_1}\right)\mathrm PQ_1+R_1,
\end{equation}
$$\Delta\mathrm PQ_2(x)+p|\Upsilon(x)|^{p-1}Q_2(x)=|x|^{2k-2}e^{U_2(x)}Z_{0,2}(x)+\left(p|\Upsilon(x)|^{p-1}-|x|^{2k-2}e^{U_2(x)}\right)\mathrm PQ_2(x)+R_2(x),$$
with
\begin{eqnarray*}
R_1(x)&:=&\left(\mathrm PQ_1(x)-Q_1(x)+\frac83\pi H(0,0)\right)e^{U_1(x)},\\
R_2(x)&:=&\left(\mathrm PQ_2(x)-Q_2(x)+\frac83k\pi H(0,0)\right)|x|^{2k-2}e^{U_2(x)};
\end{eqnarray*}
from \eqref{pq1}, \eqref{pq2}, both $R_j$'s are uniformly bounded and
\begin{equation}\label{r12}
R_1(x)=O(|x|+\alpha),\qquad\qquad\qquad R_2(x)=O\left(\left|x^k-\rho^k\right|+\beta^{2k}+\rho^k\right).
\end{equation}

We will now show that $\mathcal C_1=0$; the same argument, with minor modifications, shows that $\mathcal C_2=0$, hence concludes the proof. Multiplying \eqref{eqq1} by $\phi_n$ and integrating by parts gets
$$\int_\Omega\mathrm PQ_1h_n=\int_\Omega e^{U_1}Z_{0,1}\phi_n+\int_\Omega\left(p|\Upsilon|^{p-1}-e^{U_1}\right)\mathrm PQ_1\phi_n+\int_\Omega R_1\phi_n.$$
Let us estimate each integral: in the left-hand side, the assumption $\|h\|_*=o\left(\frac1p\right)$, \eqref{pq1} and the properties of $Q_1$ yield
\begin{eqnarray*}
\left|\int_\Omega\mathrm PQ_1h_n\right|&=&O\left(\|h\|_*\int_\Omega\frac{\alpha^{2\delta}}{\left(|x|^2+\alpha^2\right)^{1+\delta}}|\mathrm PQ_1(x)|\mathrm dx\right)\\
&=&o\left(\frac1p\int_\Omega\frac{\alpha^{2\delta}}{\left(|x|^2+\alpha^2\right)^{1+\delta}}(|Q_1(x)|+1)\mathrm dx\right)\\
&=&o\left(\frac1p\int_\Omega\frac{\alpha^{2\delta}}{\left(|x|^2+\alpha^2\right)^{1+\delta}}O\left(\log \frac1\alpha+O\left(\log \left(\frac{|x|}\alpha+2\right)\right)\right)\mathrm dx\right)\\
&=&o\left(\frac1pO\left(\log \frac1\alpha+1\right)\right)\\
&=&o(1).
\end{eqnarray*}
In the first term in the right-hand side, we just apply the dominated convergence theorem:
$$\int_\Omega e^{U_1}Z_{0,1}\phi_n\to \mathcal C_1\int_{\mathbb R^2}\frac{8\left(|y|^2-1\right)^2}{\left(|y|^2+1\right)^4}\mathrm dy=\frac83\pi \mathcal C_1.$$
Next, using Lemma \ref{derivestim}, we write:
\begin{eqnarray*}
&&\int_\Omega\left(p|\Upsilon|^{p-1}-e^{U_1}\right)\phi_n\\
&=&\int_{\left\{|x|\le\alpha^{1-\theta}\right\}}\left(p|\Upsilon|^{p-1}-e^{U_1}\right)\mathrm PQ_1\phi_n+\int_{\left\{\left|x^k-\rho^k\right|\le\beta^{k(1-\theta)}\right\}}p|\Upsilon|^{p-1}\mathrm PQ_1\phi_n\\
&+&O\left(\left(\int_{\left\{|x|>\alpha^{1-\theta}\right\}}|x|^{2k-2}e^{U_2(x)}\mathrm dx+\int_{\left\{\left|x^k-\rho^k\right|>\beta^{k(1-\theta)}\right\}}e^{U_1(x)}\mathrm dx\right)\|\mathrm PQ_1\|\|\phi_n\|_\infty\right)\\
&=&\int_{\left\{|y|\le\frac1{\alpha^\theta}\right\}}e^{U(y)}\frac1p\left(V(y)-U(y)-\frac{U(y)^2}2+O\left(\frac{\log ^4(|y|+2)}p\right)\right)\\
&&\cdot\left(\left(\frac43\log \alpha\right)Z_0(y)+O(1)\right)\phi_{n,1}(y)\mathrm dy\\
&+&\int_{\left\{|y|\le\frac1{\beta^{k\theta}}\right\}}e^{U(y)}\left(1+O\left(\frac{\log ^2(|y|+2)}p\right)\right)\left(\frac43k\log \rho+O(1)\right)\phi_{n,2}(y)\mathrm dy\\
&+&O\left(p\left(\alpha^{2(1-\theta)}\frac{\beta^{2k}}{\rho^{2k+2}}+\beta^{2k(1-\theta)}\frac{\alpha^2}{\rho^{4k}}\right)\right)\\
&=&-\left(a(\eta)+O\left(\frac1p\right)\right)\int_{\left\{|y|\le\frac1{\alpha^\theta}\right\}}e^U\left(V-U-\frac{U^2}2\right)Z_0\phi_{n,1}\\
&+&O\left(p\int_{\left\{|y|\le\frac1{\beta^{k\theta}}\right\}}e^U\phi_{n,2}\right)+O\left(pe^{-\varepsilon p}\right),
\end{eqnarray*}
where
$$a(\eta)=\frac13-\frac23\frac{k\eta^2}{(k\eta-1)(\eta+1)}\log \eta\ge -\frac16.$$
From the dominated convergence theorem, we get
$$\int_{\left\{|y|\le\frac1{\alpha^\theta}\right\}}e^U\left(V-U-\frac{U^2}2\right)Z_0\phi_{n,1}\to \mathcal C_1\int_{\mathbb R^2}\left(V-U-\frac{U^2}2\right)Z_0^2=-8\pi \mathcal C_1;$$
moreover, one may assume (up to a subsequence) that $\phi_{n,2}\to \mathcal C_2Z_0$ fast enough to have
$$\sup_{|y|\le\frac1{\beta^{k\theta}}}|\phi_{n,2}-\mathcal C_2Z_0|\le\frac1{p^2},$$
so that
\begin{eqnarray*}
&&p\int_{\left\{|y|\le\frac1{\beta^{k\theta}}\right\}}e^U\phi_{n,2}\\
&=&\mathcal C_2p\int_{\mathbb R^2}e^UZ_0+p\int_{\left\{|y|\le\frac1{\beta^{k\theta}}\right\}}e^U(\phi_{n,2}-\mathcal C_2Z_0)+p\int_{\left\{|y|>\frac1{\beta^{k\theta}}\right\}}e^UZ_0\\
&=&O\left(\frac1p+p\beta^{2\theta}\right),
\end{eqnarray*}
hence
$$\int_\Omega\left(p|\Upsilon|^{p-1}-e^{U_1}\right)\phi_n\to-8\pi a(\eta)\mathcal C_1+o(1).$$
Finally, using \eqref{r12}
\begin{eqnarray*}
\int_\Omega R_1\phi_{n,1}&=&O\left(\int_\Omega e^{U_1(x)}\left(|x|+\alpha\right)|\phi_n(x)|\mathrm dy\right)\\
&=&O\left(\alpha\int_{\mathbb R^2}e^{U(y)}(|y|+1)\mathrm dy\right)\\
&=&O(\alpha),
\end{eqnarray*}
which gives
$$8\pi\left(\frac13+a(\eta)\right)\mathcal C_1=o(1),$$
namely $\mathcal C_1=0$.
\end{proof}

\subsection{Proof of Proposition \ref{linear}: completed} 
\begin{lemma}\label{intpz} 
Let $\widetilde Z$ be the function in \eqref{def:Ztilde} and $P\widetilde Z$ its projection on $H^1_0(\Omega)$. It holds
$$\int_{\Omega} \left|\nabla P\widetilde Z\right|^2=\frac83k\pi+O\left(e^{-\varepsilon p}\right).$$
\end{lemma}
\begin{proof} 
Arguing as in the proof of Lemma \ref{pupvpw}, one gets
\begin{equation}\label{pz}
P\widetilde Z(x)=\widetilde Z(x)-8k\pi\beta^k\partial_{x_1}H(0,0)+O\left(\beta^k\left|x^k-\rho^k\right|+\beta^{3k}+\beta^k\rho^k\right),
\end{equation}
therefore
\begin{eqnarray*}
\int_{\Omega} \left|\nabla P\widetilde Z\right|^2&=&\int_\Omega|x|^{2k-2}e^{U_2(x)}\widetilde Z(x)\mathrm P\widetilde Z(x)\mathrm dx\\
&=&\int_\Omega|x|^{2k-2}e^{U_2(x)}\widetilde Z(x)\left(\widetilde Z(x)-8\pi k\beta^k\partial_{x_j}H(0,0)\right.\\
&&\left.+O\left(\beta^k\left|x^k-\rho^k\right|+\beta^{3k}+\beta^k\rho^k\right)\right)\,dx\\
&=&\int_{\frac{\Omega^k-\rho^k}{\beta^k}}ke^{U(y)}Z_1(y)\left(Z_1(y)-8k\pi\beta^k\partial_{x_1}H(0,0)\right.\\
&&\left.+O\left(\beta^{2k}|y|+\beta^{3k}+\beta^k\rho^k\right)\right)\,dy\\
&=&k\int_{\mathbb R^2}e^UZ_1^2-8\pi k^2\beta\partial_{x_1}H(0,0)\int_{\mathbb R^2}e^UZ_1+O\left(\frac{\beta^{3k}}{\rho^{3k}}+\beta^k\rho^k\right)\\
&=&\frac83k\pi+O\left(e^{-\varepsilon p}\right),
\end{eqnarray*}
where $Z_1$ is the function defined in \eqref{def:Z_i}.
\end{proof}\
 
\begin{proof}[Proof of Proposition \ref{linear}]\ Let us first prove that the \emph{a priori} estimate \eqref{estimphi} for any solution implies existence and uniqueness of solutions $\phi$ to \eqref{eqphih}, where $\phi$ satisfies the orthogonality condition \eqref{orto} and $\mathfrak c$ is as in \eqref{expressionC}.\\
Let us split $H^1_{0,k}(\Omega)$ into the direct sum of its subspaces
$$\mathbf Z:= \left\{c\mathrm P\widetilde Z:\,c\in\mathbb R\right\}\ \hbox{and}\
\mathbf Z^\perp:=\left\{\phi\in H^1_{0,k}(\Omega):\,\int_\Omega|x|^{2k-2}e^{U_2(x)}\widetilde Z(x)\phi(x)\mathrm dx=0\right\},
$$
with the respective projections denoted by $\Pi:H^1_{0,k}(\Omega)\to\mathbf Z$ and $\Pi^\perp:H^1_{0,k}(\Omega)\to\mathbf Z^\perp$.\\
Problem \eqref{eqphih} can be equivalently written as
\begin{equation}\label{eqfredholm}
(\mathrm{Id}-\mathcal K)\phi=\widetilde h,\ \hbox{ with }\ 
 \mathcal K\phi:=\Pi^\perp(-\Delta)^{-1}\left(p|\Upsilon|^{p-1}\phi\right)\ \hbox{ and }\ \widetilde h:=-\Pi^\perp(-\Delta)^{-1}h.\end{equation}
Here $(-\Delta)^{-1}:L^2_k(\Omega)\to H^1_{0,k}(\Omega)$ is the inverse of the Dirichlet Laplacian and is a compact operator, $\Pi^\perp$ is continuous and $p|\Upsilon|^{p-1}\in L^\infty(\Omega)$ (at fixed $p$), hence $\mathcal K$ is also compact on $L^2_k(\Omega)$, and from \eqref{estimphi} it is injective. Therefore, from Fredholm's alternative, \eqref{eqfredholm} has a unique solution $\phi\in\mathbf Z^\perp$ for any $\widetilde h\in\mathbf Z^\perp$; this is true in particular if $\widetilde h:=-\Pi^\perp(-\Delta)^{-1}h$ for some $h\in L^\infty_k(\Omega)$ and, by elliptic regularity $\phi\in H^1_{0,k}(\Omega)\cap W^{2,2}(\Omega)\subset H^1_{0,k}(\Omega)\cap L^\infty_k(\Omega)$.
Moreover, $\widetilde h\in\mathbf Z^\perp$ if and only if $\mathfrak c=-\frac{\int_{\Omega}\nabla h \cdot \nabla P\widetilde Z}{\int_{\Omega} \left|\nabla P\widetilde Z\right|^2}$. \\
In order to conclude the proof we show that the estimate \eqref{estimphi} holds true.\\
By Lemma \ref{orthog2} and \eqref{normbubble} 
\begin{eqnarray}
\nonumber\|\phi\|_\infty&\le&Cp\left(\|h\|_*+\left\|\mathfrak c|x|^{2k-2}e^{U_2(x)}\widetilde Z(x)\right\|_*\right)\\
\nonumber&\le&Cp\left(\|h\|_*+|\mathfrak c|\left\|\widetilde Z\right\|_\infty\left\||x|^{2k-2}e^{U_2(x)}\right\|_*\right)\\
\label{phiphc}&\le&Cp(\|h\|_*+|\mathfrak c|),
\end{eqnarray}
therefore it suffices to show that $\mathfrak c=O(\|h\|_*)$.
By linearity, this is equivalent to proving that
\[\lambda_p\leq C\|h\|_* \qquad\mbox{for any }|\mathfrak c|=\lambda_p\]
where $\lambda_p>0$ is fixed. We argue by contradiction, assuming that 
there are solutions to \eqref{eqphih} for 
\begin{equation}\label{hpAssurdoc}\|h\|_*=o(\lambda_p).\end{equation}
For our convenience let us choose $\lambda_p=\frac{1}{p}$, so that $\|\phi\|_\infty$ is bounded by \eqref{phiphc}.
%
%
Multiplying equation \eqref{eqphih} by $\mathrm P\widetilde Z$, and integrating by parts, we get
\begin{equation}\label{intphipz}
\int_\Omega p|\Upsilon|^{p-1}\phi\mathrm P\widetilde Z-\int_\Omega|x|^{2k-2}e^{U_\beta(x)}\phi(x)\mathrm P\widetilde Z(x)\mathrm dx=\int_\Omega h\mathrm P\widetilde Z+\frac{1}{p}\int_{\Omega} \left|\nabla P\widetilde Z\right|^2.
\end{equation}
From Lemma \ref{derivestim} and \eqref{pz}, on the left-hand side of \eqref{intphipz}, we deduce
\begin{eqnarray}
\nonumber&&\int_\Omega p|\Upsilon|^{p-1}\phi\mathrm P\widetilde Z-\int_\Omega|x|^{2k-2}e^{U_\beta(x)}\phi(x)\mathrm P\widetilde Z(x)\mathrm dx\\
\nonumber&=&\int_{\left\{\left|x^k-\rho^k\right|\le\beta^{k(1-\theta)}\right\}}\left(p|\Upsilon(x)|^{p-1}-|x|^{2k-2}e^{U_\beta(x)}\right)\phi(x)\mathrm P\widetilde Z(x)\mathrm dx\\
\nonumber&&+\int_\Omega|x|^{2k-2}e^{U_\beta(x)}\phi(x)\left(\mathrm P\widetilde Z(x)-\widetilde Z(x)\right)\mathrm dx\\
\nonumber&&+O\left(\|\phi\|_\infty\left\|\mathrm P\widetilde Z\right\|_\infty\int_{\left\{\left|x^k-\rho^k\right|>\beta^{k(1-\theta)}\right\}}|x|^{2k-2}e^{U_\beta(x)}\mathrm dx\right)\\
\label{intzphi}&=&\frac1p\int_{\left\{|y|\le\frac1{\beta^{k(1-\theta)}}\right\}}e^{U(y)}\left(\left(V(y)-U(y)-\frac{U(y)^2}2\right)+O\left(\frac{\log^4(|y|+2)}{p^2}\right)\right)Z_1(y)\phi_1(y)\mathrm dy\\
\nonumber&&+ O\left(\|\phi\|_\infty\beta^k\int_\Omega|x|^{2k-2}e^{U_\beta(x)}\mathrm dx+\|\phi\|_\infty\int_{\left\{|y|\le\frac1{\beta^{k(1-\theta)}}\right\}}e^U\mathrm dx\right),
\end{eqnarray}
where $Z_1$ is the function defined in \eqref{def:Z_i}.\\
Arguing similarly as in the proof of Lemma \ref{orthog2}, we can prove that the rescalement $\phi_1(y):=\phi(\alpha y)$ converges to $\mathcal CZ_0$, for some $\mathcal C\in\mathbb R$, where $Z_0$ is the function defined in \eqref{def:Z_i}. Therefore, in \eqref{intzphi}, we get
\begin{eqnarray*}
&&\int_{\left\{|y|\le\frac1{\beta^{k(1-\theta)}}\right\}}e^{U(y)}\left(\left(V(y)-U(y)-\frac{U(y)^2}2\right)+O\left(\frac{\log^4(|y|+2)}{p^2}\right)\right)Z_1(y)\phi_1(y)\mathrm dy\\
&\to&\mathcal C\int_{\mathbb R^2}e^U\left(V-U-\frac{U^2}2\right)Z_1Z_0=0,
\end{eqnarray*}
since we are integrating an odd function. Therefore, for the left-hand side of \eqref{intphipz} we get
$$\int_\Omega p|\Upsilon|^{p-1}\phi\mathrm P\widetilde Z-\int_\Omega|x|^{2k-2}e^{U_\beta(x)}\phi(x)\mathrm P\widetilde Z(x)\mathrm dx=o\left(\frac1p\right).$$
In the right-hand side of \eqref{intphipz}, using \eqref{normint}, we get 
$$\left|\int_\Omega h\mathrm P\widetilde Z\right|\le C\int_\Omega|h|\le C\|h\|_*=o\left(\frac{1}{p}\right).
$$
Hence, in view of Lemma \ref{intpz}, substituting into \eqref{intphipz}, we obtain the following contradiction
$$o\left(\frac1p\right)=\frac1p\frac83k\pi+o\left(\frac1p\right),$$
which completes the proof.
\end{proof}

\section{Proof of Theorem \ref{teop}}\label{5}

This section is committed to the last step of the Ljapunov-Schmidt reduction process. 
First, we solve a nonlinear problem via a contraction mapping argument, whose solution depends on the free parameter $\eta$ (see Proposition \ref{eqphic} below). Then, we introduce the reduced energy, and find the parameter $\eta$ minimizing it (see \eqref{def:ridotta} and Proposition \ref{reduction}). The minimum produces the solutions of our problem with the required properties (see Lemma \ref{Lemma:c=0AtCriticalPointsofF} and Subsection \ref{subsection:fine}, where the proof of Theorem \ref{teop} is completed). 

\subsection{The nonlinear problem}\label{subsection:contraction}
First of all, we solve the problem 
\begin{equation}\label{eqphic}
\left\{\begin{array}{ll}\mathscr L\phi(x)+\mathscr E(x)+\mathscr N(\phi)(x)=\mathfrak c|x|^{2k-2}e^{U_2(x)}\widetilde Z(x)&x\in\Omega\\\phi(x)=0&x\in\partial\Omega\\\int_\Omega|x|^{2k-2}e^{U_2(x)}\widetilde Z(x)\phi(x)\mathrm dx=0\end{array}\right.
\end{equation}

\begin{proposition}\label{nonlinear} 
Let $\mathscr L$, $\mathscr E$, $\mathscr N$ be the ones in \eqref{elle}, \eqref{erre} and \eqref{enne}, and let us assume that, in the definition \eqref{upsp} of $\Upsilon$, the parameters $\alpha, \beta, \rho, \tau,\eta$ satisfy \eqref{param}. 
Then there exists $C>0$ such that, for any $\eta$ satisfying \eqref{eta}, problem \eqref{eqphic}
has a unique solution $ (\phi, \mathfrak c)=\left(\phi_p(\eta), \mathfrak c_p(\eta)\right)\in H^1_{0,k}(\Omega)\cap L^\infty_k(\Omega)\times \mathbb R $, such that
\begin{equation}\label{phic}\|\phi_p(\eta)\|_{H^1_0}+\|\phi_p(\eta)\|_\infty\le\frac C{p^3}\quad \hbox{ and }\quad |\mathfrak c_p(\eta)|\le\frac C{p^4}.
\end{equation}
Moreover, the map $\eta\mapsto\phi_p(\eta)$ is of class $C^1$.
\end{proposition}

Before proving the result, we need an estimate of the second derivative of the nonlinear term $\mathscr N(\phi)=|\Upsilon+\phi|^{p-1}(\Upsilon+\phi)-|\Upsilon|^{p-1}\Upsilon-p|\Upsilon|^{p-1}\phi$. This is contained below.
\begin{lemma}\label{derivsec}
There exists $C>0$ such that, uniformly in $\Omega$,
\begin{eqnarray*}
p|\Upsilon(x)+\phi(x)|^{p-2}&\le&C\left(e^{U_1(x)}+|x|^{2k-2}e^{U_2(x)}\right)+(p\|\phi\|_\infty)^{p-2}.
\end{eqnarray*}
\end{lemma}
\begin{proof} 
Arguing as for Lemma \ref{deriv}, we get
$$p|\Upsilon(x)|^{p-2}\le C\left(e^{U_1(x)}+|x|^{2k-2}e^{U_2(x)}\right);$$
next, from the convexity of $t\mapsto|t|^{p-2}$ we get
$$|a+b|^{p-2}\le\frac1{\lambda^{p-3}}|a|^{p-2}+\frac1{(1-\lambda)^{p-3}}|b|^{p-2},\qquad\forall\,a,b\in\mathbb R,\lambda\in(0,1),$$
therefore taking $\lambda=1-\frac1p$ one gets:
\begin{eqnarray*}
p|\Upsilon(x)+\phi(x)|^{p-2}&\le&p\left(\frac1{\left(1-\frac1p\right)^{p-3}}|\Upsilon(x)|^{p-2}+p^{p-3}|\phi(x)|^{p-2}\right)\\
&\le&Cp|\Upsilon(x)|^{p-2}+p^{p-2}|\phi(x)|^{p-2}\\
&\le&C\left(e^{U_1(x)}+|x|^{2k-2}e^{U_2(x)}\right)+(p\|\phi\|_\infty)^{p-2}.
\end{eqnarray*}
\end{proof}

\begin{proof}[Proof of Proposition \ref{nonlinear}]
From Proposition \ref{linear}, one may consider the linear map $\mathcal T:(L^\infty_{k}(\Omega),\|\cdot\|_{*})\rightarrow L^\infty(\Omega)$, defined as 
$h\mapsto\mathcal Th:=\phi$, where $\phi$ is the solution to \eqref{eqphih}. We stress that we endow the space $L^\infty_k(\Omega)$ with the weighted norm $\|\cdot\|_*$ introduced in \eqref{norm}; moreover, by \eqref{estimphi}, one has that \begin{equation}\label{stimaNelCasoNonlinear}\|\mathcal Th\|_\infty\le Cp\|h\|_*,\end{equation} for some $C>0$.\\
In terms of the operator $\mathcal T$, problem \eqref{eqphic} is equivalent to the nonlinear equation
\begin{equation}\label{eqphia}
\phi=\mathcal A(\phi):=-\mathcal T(\mathscr E+\mathscr N(\phi)).
\end{equation}
Next we show that $\mathcal A: (L^\infty_k(\Omega), \|\cdot\|_{\infty})\rightarrow (L^\infty_k(\Omega), \|\cdot\|_{\infty})$
 is a contraction on a ball of radius $\frac{M}{p^3}$, for a suitable constant $M$.\\
By the Lagrange's theorem, we have the following estimates for the nonlinear term $\mathscr N(\phi)$:
\begin{eqnarray*}
|\mathscr N(\phi)|&\le&p(p-1)|\Upsilon+O(|\phi|)|^{p-2}\phi^2,\\
|\mathscr N(\phi_1)-\mathscr N(\phi_2)|&\le&p(p-1)|\Upsilon+O(|\phi_1|+|\phi_2|)|^{p-2}(|\phi_1|+|\phi_2|)|\phi_1-\phi_2|.
\end{eqnarray*}
Furthermore, from Lemma \ref{derivsec} and \eqref{normbubble}, we get
\begin{eqnarray}\label{laPrima}
\left\|p(p-1)|\Upsilon+O(|\phi|)|^{p-2}\right\|_*&\le& Cp \left(1+ p^{p-2}\|\phi\|_{\infty}^{p-2}\right),\\\label{laSeconda}
\left\|p(p-1)|\Upsilon+O(|\phi_1|+|\phi_2|)|^{p-2}\right\|_*&\le&  Cp\left(1+ p^{p-2}(\|\phi_1\|_{\infty}+\|\phi_2\|_\infty)^{p-2}\right).
\end{eqnarray}
Therefore
\begin{eqnarray*}
\|\mathscr N(\phi)\|_*&\le&Cp\|\phi\|_\infty^2 \left(1+  p^{p-2}\|\phi\|_{\infty}^{p-2}\right)\\
\|\mathscr N(\phi_1)-\mathscr N(\phi_2)\|_*&\le&Cp(\|\phi_1\|_\infty+\|\phi_2\|_\infty)\|\phi_1-\phi_2\|_\infty \left(1+p^{p-2}(\|\phi_1\|_\infty+\|\phi_2\|_\infty)^{p-2}\right).
\end{eqnarray*}
Putting together the definition of $\mathcal A$ in \eqref{eqphia}, \eqref{stimaNelCasoNonlinear}, Proposition \ref{r} and \eqref{laPrima}, we get, for $\|\phi\|_\infty\le\frac M{p^3}$:
$$\|\mathcal A(\phi)\|_\infty\le Cp(\|\mathscr E\|_*+\|\mathscr N(\phi)\|_*)\le\frac{C^2}{p^3}+\frac{C^2M^2}{p^4}+   \frac{C^2M^p}{p^{2p}}  ,$$
which is smaller than $\frac M{p^3}$, if $M$ is strictly larger than $C^2$ and $p$ is large enough.\\
Similarly, if $\|\phi_i\|_\infty\le\frac M{p^3}$ for $i=1,2$, then from \eqref{laSeconda}:
\begin{eqnarray*}
\|\mathcal A(\phi_1)-\mathcal A(\phi_2)\|_\infty&\le&Cp\|\mathscr N(\phi_1)-\mathscr N(\phi_2)\|\\&\le&\frac{2C^2M}p\|\phi_1-\phi_2\|_\infty+\frac{C^2(2M)^{p-1}}{p^{2p-3}} \|\phi_1-\phi_2\|_\infty\\
&\le&\frac{\|\phi_2-\phi_2\|_\infty}2,
\end{eqnarray*}
for $p$ large enough.
Therefore, $\mathcal A$ is a contraction on a ball of radius $\frac M{p^3}$.\\
As a consequance, $\mathcal A$ has a fixed point $\phi\in L^\infty_k(\Omega)$,  which solves \eqref{eqphia}, hence \eqref{eqphic} and such that $\|\phi\|_{\infty}\leq M/p^3$.\\
Furthermore, as in the proof of Proposition \ref{linear}, we get
$$|\mathfrak c|\le C\left(\|h\|_*+\frac{\|\phi\|_\infty}p\right),$$
so that
$$|\mathfrak c|\le C\left(\|\mathscr E\|_*+\|\mathscr N(\phi)\|_*+\frac{\|\phi\|_\infty}p\right)\le\frac C{p^4}$$
and, from Remark \ref{normphi}, we get
$$\|\phi\|_{H^1_0}\le C\left(\|\phi\|_\infty+\|\mathscr E\|_*+\|\mathscr N(\phi)\|_*\right)\le\frac C{p^3}.$$
The regularity of $\eta\mapsto\phi$ follows by a standard application of the implicit function theorem (see for instance \cite[Lemma 4.1]{EspositoMussoPistoiaJDE2006}).
\end{proof}

\subsection{The reduced energy}\label{subsection:reduced}
Let us introduce the functional 
$$J_p(u):=\int_\Omega\left(\frac12|\nabla u|^2-\frac1{p+1}|u|^{p+1}\right),\ u\in H^1_0(\Omega)$$
whose critical points are solutions to \eqref{eqp}. Let  $\Upsilon$ be as in \eqref{upsp}, where the parameters $\alpha, \beta, \rho, \tau$ are chosen as in  \eqref{param}. If the free parameter $\eta$ satisfies \eqref{eta}, then Proposition \ref{nonlinear} holds, hence the solution $(\phi_p(\eta),\mathfrak c_p(\eta))$ of the nonlinear problem \eqref{eqphic} is well defined. Then we set the \emph{reduced energy} as
\begin{equation}\label{def:ridotta}F_p(\eta):=J_p\left(\Upsilon+\phi_p(\eta)\right).\end{equation}
Our aim is to find $\eta=\eta(p)$ such that  $ \mathfrak c_p(\eta)=0.$
This is equivalent to find a critical point of the reduced energy, as stated in the following Lemma.

\begin{lemma} \label{Lemma:c=0AtCriticalPointsofF}
The functional $\eta\mapsto F_p(\eta)$ is of class $C^1$ and, for $p$ large enough, if $F_p'(\eta)=0$ then $\mathfrak c_p(\eta)=0$.
\end{lemma}

\begin{proof} 
The smoothness of $F_p$ follows from the smoothness of $\eta\mapsto\phi_p(\eta)$, namely the last statement of Proposition \ref{nonlinear}.\\ 
If $F_p'(\eta)=0$, then by differentiating under integral sign we have
\begin{eqnarray*}
0&=&\int_\Omega(\Delta(\Upsilon+\phi)+|\Upsilon+\phi|^{p-1}(\Upsilon+\phi))(\partial_\eta\Upsilon+\partial_\eta\phi)\\
&=&-\mathfrak c(\eta)\int_\Omega|x|^{2k-2}e^{U_2(x)}\widetilde Z(x)(\partial_\eta\Upsilon(x)+\partial_\eta\phi(x))\mathrm dx\\
&=&-\mathfrak c(\eta)\int_\Omega|x|^{2k-2}e^{U_2(x)}\widetilde Z(x)\partial_\eta\Upsilon(x)\mathrm dx\\
&+&\mathfrak c(\eta)\int_\Omega|x|^{2k-2}\partial_\eta\left(e^{U_2(x)}\widetilde Z(x)\right)\phi(x)\mathrm dx,
\end{eqnarray*}
since $\int_\Omega|x|^{2k-2}e^{U_2(x)}\widetilde Z(x)\phi\equiv0$.\\
By the explicit definition \eqref{upsp} of $\Upsilon$ we get
\begin{eqnarray*}
\partial_\eta\Upsilon(x)&=&\partial_\eta\tau\left(\mathrm PU_1+\frac{\mathrm PV_1}p+\frac{\mathrm PW_1}{p^2}\right)-\left(\partial_\eta\tau\eta+\tau\right)\left(\mathrm PU_2+\frac{\mathrm PV_2}p+\frac{\mathrm PW_2}{p^2}\right)\\
&+&\tau\left(\mathrm P\nabla U\left(\frac x\alpha\right)+\frac{\mathrm P\nabla V\left(\frac x\alpha\right)}p+\frac{\mathrm P\nabla W\left(\frac x\alpha\right)}{p^2}\right)\cdot\left(-\frac x{\alpha^2}\right)\partial_\eta\alpha\\
&+&\tau\eta\left(\mathrm P\nabla U\left(\frac{x^k-\rho^k}{\beta^k}\right)+\frac{\mathrm P\nabla V\left(\frac{x^k-\rho^k}{\beta^k}\right)}p+\frac{\mathrm P\nabla W\left(\frac{x^k-\rho^k}{\beta^k}\right)}{p^2}\right)\cdot\\
&&\cdot\left(-\frac{x^k-\rho^k}{\beta^{2k}}\partial_\eta\left(\beta^k\right)-\frac{\partial_\eta\left(\rho^k\right)}{\beta^k}\right)\\
&=&O\left(\frac1p\left(\log \frac1\alpha+\log \frac1{\beta^k}\right)\right)+O\left(\frac1p\mathrm P(O(1))\left(\frac{\partial_\eta\alpha}\alpha+\frac{\partial_\eta\left(\beta^k\right)}{\beta^k}\right)\right)\\
&-&\tau\eta\mathrm P\left(-\frac{\frac{x^k-\rho^k}{\beta^k}}{\left|\frac{x^k-\rho^k}{\beta^k}\right|^2+1}\right)\left(1+O\left(\frac1p\right)\right)\cdot\frac{\partial_\eta\left(\rho^k\right)}{\beta^k}\\
&=&\frac{\tau\eta}4\mathrm P\widetilde Z(x)k\frac{\partial_\eta\rho}\rho\frac{\rho^k}{\beta^k}\left(1+O\left(\frac1p\right)\right)\\
&=&b(\eta)\mathrm P\widetilde Z(x)\frac{\rho^k}{\beta^k}\left(1+O\left(\frac{\log p}p\right)\right)+O(1),
\end{eqnarray*}
with
$$b(\eta):=\frac{\sqrt e}8\eta^\frac{\eta(k-1)-1}{(k\eta-1)(\eta+1)}kK_1(\eta)\frac{(k\eta-1)(\eta+1)-\left(k\eta^2+1\right)\log \eta}{(k\eta-1)^2(\eta+1)^2}$$
uniformly bounded from above and below by positive constants and $K_1(\eta)$ as in the Lemma \ref{asymptotic}; we used that $\nabla U(y)\cdot y,\nabla V(y)\cdot y,\nabla W(y)\cdot y$ are all bounded, estimates \eqref{pu1} and \eqref{pu2} and Lemma \ref{asymptotic}.\\
On the other hand, since $e^{U(y)}|Z_1(y)|+\left|\nabla\left(e^{U(y)}Z_1(y)\right)\cdot y\right|=O\left(\frac1{|y|^5+1}\right)$, then
\begin{eqnarray*}
&&\partial_\eta\left(e^{U_2(x)}\widetilde Z(x)\right)\\
&=&\partial_\eta\left(\frac1{\beta^{2k}}e^{U\left(\frac{x^k-\rho^k}{\beta^k}\right)}Z_1\left(\frac{x^k-\rho^k}{\beta^k}\right)\right)\\
&=&-\frac{\partial_\eta\left(\beta^k\right)}{\beta^{3k}}e^{U\left(\frac{x^k-\rho^k}{\beta^k}\right)}Z_1\left(\frac{x^k-\rho^k}{\beta^k}\right)\\
&+&\frac1{\beta^{2k}}\partial_\eta\left(e^{U\left(\frac{x^k-\rho^k}{\beta^k}\right)}Z_1\left(\frac{x^k-\rho^k}{\beta^k}\right)\right)\cdot\left(-\frac{x^k-\rho^k}{\beta^{2k}}\partial_\eta\left(\beta^k\right)-\frac{\partial_\eta\left(\rho^k\right)}{\beta^k}\right)\\
&=&O\left(p\frac{\rho^k}{\beta^{3k}}\frac{\beta^{5k}}{\left|x^k-\rho^k\right|^5+\beta^{5k}}\right).
\end{eqnarray*}
Therefore, in view of Lemma \ref{intpz},
\begin{eqnarray*}
0&=&\mathfrak c(\eta)b(\eta)\frac{\rho^k}{\beta^k}\int_\Omega|x|^{2k-2}e^{U_2(x)}\widetilde Z(x)\mathrm P\widetilde Z(x)\mathrm dx\left(1+O\left(\frac{\log p}p\right)\right)\\
&+&O\left(\mathfrak c(\eta)\int_\Omega|x|^{2k-2}e^{U_2(x)}\widetilde Z(x)\mathrm dx\right)\\
&+&O\left(\mathfrak c(\eta)p\frac{\rho^k}{\beta^{3k}}\int_\Omega|x|^{2k-2}e^{U_2(x)}\widetilde Z(x)\frac{\beta^{5k}}{\left|x^k-\rho^k\right|^5+\beta^{5k}}|\phi(x)|\mathrm dx\right)\\
&=&\mathfrak c(\eta)b(\eta)\frac{32k}3\pi\frac{\rho^k}{\beta^k}\left(1+O\left(\frac{\log p}p\right)\right)\\
&+&O\left(\mathfrak c(\eta)\beta^{3k}\int_\Omega|x|^{2k-2}\frac{\left|x^k-\rho^k\right|}{\left|x^k-\rho^k\right|^6+\beta^{6k}}\mathrm dx\right.\\
&&\left.+\mathfrak c(\eta)p\rho^k\beta^{5k}\|\phi\|_\infty\int_\Omega|x|^{2k-2}\frac{\left|x^k-\rho^k\right|}{\left|x^k-\rho^k\right|^9+\beta^{9k}}\mathrm dx\right)\\
&=&\mathfrak c(\eta)b(\eta)\frac{32k}3\pi\frac{\rho^k}{\beta^k}\left(1+O\left(\frac{\log p}p\right)\right)+O(\mathfrak c(\eta))+O\left(\mathfrak c(\eta)\frac1{p^2}\frac{\rho^k}{\beta^k}\right)
\end{eqnarray*}
that is, for $p$ large (independently on $\eta$), $\mathfrak c(\eta)=0$.
\end{proof}\

\subsection{Minimizing the reduced energy}\label{subsection:minimo}
The main order term of $F_p$ is given by the following lemma:

\begin{lemma}\label{energy} 
For any $\epsilon>0$ small enough, it holds true that
$$F_p(\eta)=\frac{4\pi}{ep}\left[\varphi_k(\eta)+O\left(\frac{\log p}{p }\right)\right]\ \hbox{as}\ p\to+\infty \ 
\hbox{uniformly in}\ \eta\in[\eta_{\infty,k}-\epsilon,\eta_{\infty,k}+\epsilon].$$
Here
$$\varphi_k(\eta):=\left(1+k\eta^2\right)\eta^{-\frac{2k\eta^2}{(k\eta-1)(\eta+1)}}.$$

\end{lemma}

\begin{proof} 
Multiplying \eqref{eqphic} by $\Upsilon+\phi$ and integrating by parts we get
\begin{eqnarray*}
\int_\Omega|\Upsilon+\phi|^p&=&\int_\Omega|\nabla(\Upsilon+\phi)|^2+\mathfrak c\int_\Omega|x|^{2k-2}e^{U_2(x)}\widetilde Z(x)(\Upsilon(x)+\phi(x))\mathrm dx\\
&=&\int_\Omega|\nabla(\Upsilon+\phi)|^2+O\left(\mathfrak c\int_\Omega|x|^{2k-2}e^{U_2(x)}\widetilde Z(x)\Upsilon(x)\mathrm dx\right)\\
&=&\int_\Omega|\nabla(\Upsilon+\phi)|^2+O\left(\frac1{p^4}\right),
\end{eqnarray*}
in view of estimate \eqref{phic}; therefore,
\begin{eqnarray}
\nonumber F_p(\eta)&=&\left(\frac12-\frac1{p+1}\right)\int_\Omega|\nabla(\Upsilon+\phi)|^2+O\left(\frac1{p^5}\right)\\
\nonumber&=&\left(\frac12+O\left(\frac1p\right)\right)\left(|\nabla\Upsilon|^2+2\nabla\Upsilon\cdot\nabla\phi+|\nabla\phi|^2\right)+O\left(\frac1{p^5}\right)\\
\label{estimfp}&=&\left(\frac12+O\left(\frac1p\right)\right)\int_\Omega|\nabla\Upsilon|^2+O\left(\frac1{p^{3}}\sqrt{\int_\Omega|\nabla\Upsilon|^2}\right)+O\left(\frac1{p^5}\right).
\end{eqnarray}
As for the main term, using equations \eqref{liouvilleq} and \eqref{liouvillesing}, Lemmas \ref{pupvpw} and \ref{asymptotic} and the relations \eqref{param}, we get
\begin{eqnarray*}
&&\int_\Omega|\nabla\Upsilon|^2\\
&=&\tau^2\left(\int_\Omega e^{U_1(x)}\left(1+\frac{f_1\left(U\left(\frac x\alpha\right)\right)}p+\frac{f_2\left(U\left(\frac x\alpha\right),V\left(\frac x\alpha\right)\right)}{p^2}\right)\mathrm PU_1(x)\mathrm dx\right.\\
&&\left.-2\eta \int_\Omega e^{U_1(x)}\left(1+\frac{f_1\left(U\left(\frac x\alpha\right)\right)}p+\frac{f_2\left(U\left(\frac x\alpha\right),V\left(\frac x\alpha\right)\right)}{p^2}\right)\mathrm PU_2(x)\mathrm dx\right.\\
&&\left.+\eta^2\int_\Omega|x|^{2k-2}e^{U_2(x)}\left(1+\frac{f_1\left(U\left(\frac{x^k-\rho^k}{\beta^k}\right)\right)}p+\frac{f_2\left(U\left(\frac{x^k-\rho^k}{\beta^k}\right),V\left(\frac{x^k-\rho^k}{\beta^k}\right)\right)}{p^2}\right)\mathrm PU_2(x)\mathrm dx\right.\\
&&\left.+O\left(\int_\Omega\frac{|\nabla\mathrm PU_1|^2+|\nabla\mathrm PV_1|^2+|\nabla\mathrm PW_1|^2+|\nabla\mathrm PU_2|^2+|\nabla\mathrm PV_2|^2+|\nabla\mathrm PW_2|^2}p\right)\right)\\
&=&\left(\frac e{p^2}\eta^{-\frac{2k\eta^2}{(k\eta-1)(\eta+1)}}+O\left(\frac{\log p}{p^3}\right)\right)\cdot\left(\int_\Omega e^{U_1(x)}\left(1+O\left(\frac{\log ^4\left(\left|\frac x\alpha\right|+2\right)}p\right)\right)\mathrm PU_1(x)\mathrm dx\right.\\
&&\left.-2\eta\int_{\left\{|x|\le\frac\rho2\right\}}e^{U_1(x)}\left(1+O\left(\frac{\log ^4\left(\left|\frac x\alpha\right|+2\right)}p\right)\right)\mathrm PU_2(x)\mathrm dx\right.\\
&+&\left.\eta^2\int_\Omega|x|^{2k-2}\left(1+O\left(\frac{\log ^4\left(\left|\frac{x^k-\rho^k}{\beta^k}\right|+2\right)}p\right)\right)e^{U_2(x)}\mathrm PU_2(x)\mathrm dx\right.\\
&&\left.+O\left(\int_{\left\{|x|>\frac\rho2\right\}}\frac{\alpha^2}{|x|^4}\log ^4\left(\left|\frac x\alpha\right|+2\right)\log \frac1\beta\mathrm dx\right)\right.\\
&&+O\left(\frac1{p^3}\int_\Omega\left(\frac{|x|^2}{\left(|x|^2+\alpha^2\right)^2}+\frac{\beta^{2k}}{\left(\left|x^k-\rho^k\right|^2+\beta^{2k}\right)^2}\right)\mathrm dx\right)\\
&=&\left(\frac e{p^2}\frac{\varphi_k(\eta)}{1+k\eta^2}+O\left(\frac{\log p}{p^3}\right)\right)\cdot\left(\int_\Omega e^{U(y)}\left(1+\frac{\log ^4(|y|+2)}p\right)(-4\log \alpha+O(\log (|y|+2)))\mathrm dy\right.\\
&&\left.-2\eta\int_{\left\{|y|\le\frac12\frac\rho\alpha\right\}}e^{U(y)}\left(1+\frac{\log ^4(|y|+2)}p\right)(-4\log \rho+O(1))\mathrm dy\right.\\
&&\left.+\eta^2\int_\Omega ke^{U(y)}\left(1+\frac{\log ^4(|y|+2)}p\right)(-4k\log \beta+O(\log (2+|y|)))\mathrm dy\right.\\
&&\left.+O\left(\frac1{p^2}\frac{\alpha^2}{\rho^2}\log ^4\left(\frac\rho\alpha\right)\log \frac1\beta\right)\right)+O\left(\frac{\log\frac1\alpha+\log\frac1\beta}{p^3}\right)\\
&=&\left(\frac e{p^2}\frac{\varphi_k(\eta)}{1+k\eta^2}+O\left(\frac{\log p}{p^3}\right)\right)\left(8\pi\left(1+O\left(\frac{\alpha^2}{\rho^2}+\frac{\beta^{2k}}{\rho^{2k}}\right)\right)\right.\\
&&\left.\cdot\left(-4\log \alpha+8\eta\log \rho-4k^2\eta^2\log \beta\right)+O(1)\right)+O\left(\frac\alpha\rho\right)+O\left(\frac1{p^2}\right)\\
&=&\left(\frac e{p^2}\frac{\varphi_k(\eta)}{1+k\eta^2}+O\left(\frac{\log p}{p^3}\right)\right)8\pi\left(\left(1+k\eta^2\right)p+O(1)\right)+O\left(\frac\alpha\beta+\frac{\beta^{2k}}{\rho^{2k}}+\frac1{p^2}\right)\\
&=&\frac{8\pi e}p\varphi_k(\eta)+O\left(\frac{\log p}{p^2}\right);
\end{eqnarray*}
one concludes putting together with \eqref{estimfp}.
\end{proof}

We are now in position to prove solve the \emph{reduced problem} and to prove the main result of this section.

\begin{proposition}\label{reduction} 
For any $\epsilon>0$ there exists $p_\epsilon>0$ such that for any $p>p_\epsilon$ the function $F_p$ has a critical point $\eta=\eta(p)$ satisfying \eqref{eta}, namely
\begin{equation}\label{asympteta}
\eta(p)\in(\eta_{\infty,k}-\epsilon,\eta_{\infty,k}+\epsilon),\ \eta_{\infty,k}:=\frac2{k-1}.
\end{equation}
In particular, $\eta(p)\to\eta_{\infty,k}$ as $p\to+\infty.$
\end{proposition} 
 
\begin{proof} 
 {Thanks to Lemma \ref{energy}, $F_p$ converges to $\varphi_k$ uniformly, up to a suitable rescaling. The function $\varphi_k$ has a strict (because non-degenerate) minimum point at the point $\eta_{\infty,k}=\frac2{k-1}$. Therefore, for any fixed $\epsilon>0$ if $p$ is large enough the function $F_p$ has a minimum point $\eta(p) \in(\eta_{\infty,k}-\epsilon,\eta_{\infty,k}+\epsilon) .$ The claim immediately follows.} \end{proof}\

\subsection{Proof of Theorem \ref{teop}: completed}
\label{subsection:fine}
From Proposition \ref{nonlinear}, we are able to solve problem \eqref{eqphic} for some $\mathfrak c\in\mathbb R$, and from Proposition \ref{reduction} and Lemma \ref{Lemma:c=0AtCriticalPointsofF} we are able to take $\mathfrak c=0$ for a suitable choice of $\eta$. Therefore, we got a solution $\phi$ to \eqref{rln}, that is a solution to \eqref{eqp} in the form \eqref{uphi}.\\
{By Lemma \ref{asymptotic}, taking into account that $\eta=\eta_{\infty,k}+o(1)$ as $p\to+\infty$ 
we immediately deduce the rate of the parameters $\alpha,$ $\beta$, $\rho$ and $\tau$ in terms of $p $.}
\\
Next, arguing as in Lemma \ref{out}, one gets
\begin{eqnarray*}
|x|>\alpha&\Rightarrow&p\Upsilon(x)_+^p\le C\frac{\alpha^{4\delta}}{|x|^{2+4\delta}},\\
\left|x^k-\rho^k\right|>\beta^k&\Rightarrow&p\Upsilon(x)_-^p\le C\frac{\beta^{4k\delta}|x|^{2k-2}}{\left|x^k-\rho^k\right|^{2+4\delta}}.
\end{eqnarray*}
By \eqref{pu1}, \eqref{pu2} and Lemma \ref{asymptotic}, $\Upsilon$ is uniformly bounded in $L^\infty(\Omega)$, therefore the same estimates hold true for $p\Upsilon(x)_\pm^{p+1}$.\\
From the convexity of $t\to|t|^{p+1}$ one gets, as in the proof of Lemma \ref{derivsec},
$$|a+b|^{p+1}\le\frac1{\lambda^p}|a|^{p+1}+\frac1{(1-\lambda)^p}|b|^{p+1},\qquad\forall\,a,b\in\mathbb R,\lambda\in(0,1)$$
we get
\begin{eqnarray*}
p|u(x)|^{p+1}&\le&p\frac1{t^p}|\Upsilon(x)|^{p+1}+p\frac1{(1-t)^p}|\phi(x)|^{p+1}\\
&\le&C\frac1{t^p}e^{-4\varepsilon\delta p}\left(\frac1{|x|^{2+4\delta}}+\frac{|x|^{2k-2}}{\left|x^k-\rho^k\right|^{2+4\delta}}\right)+p\frac1{(1-t)^p}\left(\frac C{p^3}\right)^{p+1};
\end{eqnarray*}
choosing $t=e^{-2\varepsilon\delta}$, one has $p|u|^{p+1}\to0$ in $L^\infty_{\mathrm{loc}}(\Omega\setminus\{0\})$.\\
Similarly, choosing $t=\alpha^\frac{2\theta\delta}p$, one has
$$pu(x)_+^{p+1}\le p\Upsilon(x)_+^{p+1}+p|\phi(x)|^{p+1}\le C\frac{\alpha^{(4-2\theta)\delta}}{|x|^{2+4\delta}}+\left(\frac C{p^3}\right)^{p+1},$$
which gives
$$\int_{\left\{|x|>\alpha^{1-\theta}\right\}}pu(x)_+^{p+1}\mathrm dx\le C\alpha^{2\theta\delta}+\left(\frac C{p^3}\right)^p\to0;$$
on the other hand, using \eqref{estimuvw}, \eqref{param}, Lemma \ref{asymptotic} and \eqref{asympteta},
\begin{eqnarray*}
&&\int_{\left\{|x|\le\alpha^{1-\theta}\right\}}pu(x)_+^{p+1}\mathrm dx\\
&=&\int_{\left\{|y|\le\frac1{\alpha^\theta}\right\}}\alpha^2pu(\alpha y)_+^{p+1}\mathrm dy\\
&=&\int_{\left\{|y|\le\frac1{\alpha^\theta}\right\}}\alpha^2p\left(\tau\left(p+U(y)+\frac{\log (|y|+2)}p\right)\right)^{p+1}\mathrm dy\\
&=&\tau^2p^2\int_{\left\{|y|\le\frac1{\alpha^\theta}\right\}}\left(1+\frac{U(y)}p+O\left(\frac{\log (|y|+2)}{p^2}\right)\right)^{p+1}\mathrm dy\\
&=&\left(e\eta^{-\frac{2k\eta^2}{(k\eta-1)(\eta+1)}}+O\left(\frac{\log p}p\right)\right)\int_{\left\{|y|\le\frac1{\alpha^\theta}\right\}}e^{\left(1+\frac1p\right)U(y)}\left(1+O\left(\frac{\log ^2(|y|+2)}p\right)\right)\mathrm dy\\
&\to&8\pi e\left(\frac{k-1}2\right)^\frac{8k}{(k+1)^2}.
\end{eqnarray*}
By essentially the same argument one gets
\begin{eqnarray*}
\int_{\left\{\left|x^k-\rho^k\right|>\beta^{k(1-\theta)}\right\}}pu(x)_+^{p+1}\mathrm dx&\to&0,\\
\int_{\left\{\left|x^k-\rho^k\right|\le\beta^{k(1-\theta)}\right\}}pu(x)_-^{p+1}\mathrm dx&\to&8\pi ek\left(\frac2{k-1}\right)^\frac{2(k-1)^2}{(k+1)^2}.
\end{eqnarray*}
That concludes the proof.

\section*{Acknowledgments}
The authors wish to thank Sergio Cruz-Bl\'azquez, Francesca Gladiali and Massimo Grossi for many useful discussions.

\bibliography{bcip}
\bibliographystyle{abbrv}

\end{document}